\documentclass[12pt]{amsart}
\usepackage{amssymb}
\usepackage{amsmath}
\usepackage{longtable}
\newcommand\g{{\mathfrak g}}
\newcommand\h{{\mathfrak h}}

\renewcommand{\b}{\mathfrak{b}}
\newcommand\gl{\mathfrak{gl}}
\renewcommand{\a}{\mathfrak{a}}

\newcommand\m{\mathfrak m}
\renewcommand\l{\mathfrak l}
\newcommand\n{\mathfrak n}
\newcommand\q{\mathfrak q}
\newcommand\z{\mathfrak z}
\newcommand\s{\mathfrak s}

\renewcommand\u{{\mathfrak u}}
\newcommand\skewperp{\angle}
\renewcommand{\t}{\mathfrak{t}}

\newcommand\tr{\operatorname{tr}}

\newcommand\im{\operatorname{im}}
\newcommand\Spec{\operatorname{Spec}}

\newcommand\Aut{\operatorname{Aut}}

\newcommand\C{\mathbb C}
\newcommand\X{\mathfrak X}

\newcommand\R{\mathbb R}

\newcommand\Z{\mathbb Z}
\newcommand\A{\mathfrak A}

\newcommand\Rad{\operatorname{R}}

\renewcommand\O{\operatorname{O}}
\renewcommand\sl{\mathfrak{sl}}
\newcommand\so{\mathfrak{so}}
\renewcommand\sp{\mathfrak{sp}}
\newcommand\spin{\mathfrak{spin}}
\newcommand\GL{\mathop{\rm GL}\nolimits}

\newcommand\SL{\mathop{\rm SL}\nolimits}
\newcommand\Sp{\mathop{\rm Sp}\nolimits}
\newcommand\SO{\mathop{\rm SO}\nolimits}

\newcommand\Span{\operatorname{Span}}

\newcommand\Hom{\operatorname{Hom}}
\newcommand{\ad}{\mathop{\rm ad}\nolimits}
\newcommand{\Ad}{\mathop{\rm Ad}\nolimits}
\renewcommand{\mod}{\mathop{\rm mod}\nolimits}
\newcommand{\rank}{\mathop{\rm rk}\nolimits}
\newcommand{\cork}{\mathop{\rm cork}\nolimits}

\renewcommand{\Ad}{\mathop{\rm Ad}\nolimits}

\newcommand\quo{/\!/}
\newtheorem{Thm}{Theorem}[subsection]
\newtheorem{Prop}[Thm]{Proposition}
\newtheorem{Cor}[Thm]{Corollary}
\newtheorem{Lem}[Thm]{Lemma}
\theoremstyle{definition}
\newtheorem{Ex}[Thm]{Example}
\newtheorem{defi}[Thm]{Definition}
\newtheorem{Rem}[Thm]{Remark}
\newtheorem{Alg}[Thm]{Algorithm}
\unitlength=1mm   \numberwithin{equation}{section}
\numberwithin{table}{section} \oddsidemargin=0cm
\author{Ivan V. Losev}
\title{Computation of  weight lattices of  $G$-varieties}
\begin{document}
\begin{abstract}
Let $G$ be a connected reductive group. To any irreducible
$G$-variety $X$ one assigns the lattice generated by all weights
of $B$-semiinvariant rational functions on $X$, where $B$ is a
Borel subgroup of $G$. This lattice is called the weight lattice
of $X$. We establish algorithms for computing weight lattices for
homogeneous spaces and affine homogeneous vector bundles. For
affine homogeneous spaces of rank $\rank G$ we present a more or
less explicit computation.
\end{abstract}
\maketitle \tableofcontents
\section{Introduction}
Throughout the paper the base field is $\C$.

Let $G$ be a connected reductive group  and $X$ be a normal
irreducible $G$-variety. Choose a Borel subgroup $B\subset G$ and a
maximal torus $T\subset B$.

Consider the action of $G$ on the field of rational functions
$\C(X)$. The set of $B$-\!\! eigencharacters of this action form a
subgroup in the character lattice $\X(B)$ of $B$. This subgroup is
called the {\it weight lattice} of $X$ and denoted by $\X_{G,X}$. In
other words,
$$\X_{G,X}=\{\lambda\in \X(B)| \exists f_\lambda\in \C(X), b.f_\lambda=\lambda(b)f_\lambda\}.$$
The weight lattice is an important invariant of a $X$. It is used,
for instance, in the equivariant embedding theory of Luna and Vust,
\cite{LV}. By the rank  of $X$ (denoted by $\rank_G(X)$) we mean the
rank of $\X_{G,X}$.

The goal of this paper is compute or, rather, develop an
algorithm computing the lattice $\X_{G,X}$ for some classes of
$G$-varieties. The classes in interest are homogeneous spaces $G/H$
($H$ is an algebraic subgroup of $G$) and affine homogeneous vector
bundle $G*_HV$ ($H$ is a reductive subgroup of $G$ and $V$ is an
$H$-module).

The basic result in the computation of the weight lattices for
$G$-varieties is the reduction procedure due to Panyushev,
\cite{Panyushev2}. It reduces the computation of the weight lattice
for $X$ to that for an {\it affine} homogeneous space (in fact,
together with some auxiliary datum~-- a point from  a so-called
distinguished component). For affine homogeneous spaces $X$ the
spaces $\a_{G,X}:=\X_{G,X}\otimes_\Z\C$ were computed in
\cite{ranks}. Moreover, in \cite{comb_ham} the author found a way to
reduce the computation of $\X_{G,X}$ to the case when
$\rank_G(X)=\rank(G)$. We would like to note that this reduction
traces back to another paper of Panyushev, \cite{Panyushev3}.

So, essentially, the only original result of this paper is the
computation of $\X_{G,X}$ for  affine homogeneous spaces $X=G/H$
with $\rank_G(X)=\rank(G)$. The main result here is very technical
Theorem \ref{Thm:7.0.5}.

The most important step in the computation of $\X_{G,G/H}$ is
computing a smaller  lattice, the {\it root lattice}
$\Lambda_{G,G/H}$ established by Knop, \cite{Knop8}. In a sense,
$\Lambda_{G,G/H}$ is an "essential" part of $\X_{G,G/H}$. The
advantage of $\Lambda_{G,G/H}$ over $\X_{G,G/H}$ is a much better
behavior. For example, $\Lambda_{G,G/H}$ depends only on the Lie
algebra $\h$ of $H$. Further, it is generated by a certain root
system in $\a_{G,G/H}$. The Weyl group of that root system is the so
called {\it Weyl group} $W_{G,G/H}$ of $G/H$. All groups $W_{G,G/H}$
were computed in author's preprint \cite{Weyl}. The main result of
the computation of $\Lambda_{G,G/H}$, Theorem \ref{Thm:7.0.2} is
much less technical than Theorem \ref{Thm:7.0.5}.

As in the computation of Weyl groups in \cite{Weyl}, the main
ingredient in the computation of the lattices $\Lambda_{G,G/H}$ is
the theory of Hamiltonian actions of reductive groups developed in
\cite{Knop1},\cite{slice},\cite{alg_hamil},\cite{comb_ham},\cite{fibers}.

Let us describe briefly the content of the paper. In Section
\ref{SECTION_notation} we introduce conventions and the list of
notation used in the paper. In Section \ref{SECTION_prelim} we
review some known results and constructions related to weight and
root lattices, including reductions of
\cite{Panyushev2},\cite{comb_ham} discussed above. Section
\ref{SECTION_Ham} is devoted to the study of the root lattices of
Hamiltonian actions. In   Section \ref{SECTION_root} we state and
prove our main results, Theorems \ref{Thm:7.0.2},\ref{Thm:7.0.5}.
Finally, in Section \ref{SECTION_algorithm} we briefly quote an
algorithm for computing the weight lattices of homogeneous spaces
and affine homogeneous vector bundles. Each of Sections
\ref{SECTION_prelim}-\ref{SECTION_root} is divided into subsections,
the first subsection of each of these sections describes its content
in more detail.

\section{Notation and conventions}\label{SECTION_notation}
For an algebraic group denoted by a capital Latin letter we denote
its Lie algebra by the  corresponding small German letter. For
example, the Lie algebra of  $\widetilde{L}_{0}$ is denoted by
$\widetilde{\l}_0$.

{\bf $H$-morphisms, $H$-subvarieties, etc.} Let $H$ be an algebraic
group. We say that a variety $X$ is an $H$-variety if an action of
$H$ on $X$ is given. By an $H$-subset (resp., subvariety) in a given
$H$-variety we mean an $H$-stable subset (resp., subvariety). A
morphism  of $H$-varieties is said to be an $H$-morphism if it is
$H$-equivariant.

{\bf Borel subgroups and maximal tori.} While considering a
reductive group $G$, we always fix its Borel subgroup $B$ and a
maximal torus $T\subset B$. In accordance with this choice, we fix
the root system $\Delta(\g)$ and the system of simple roots
$\Pi(\g)$ of $\g$. Let $U$ denote the unipotent radical of $B$ so that
$B=T\rightthreetimes U$. 

If $G_1,G_2$ are reductive groups with  fixed Borel subgroups
$B_i\subset G_i$ and maximal tori $T_i\subset B_i$, then we take
$B_1\times B_2,T_1\times T_2$ for the fixed Borel subgroup and
maximal torus in $G_1\times G_2$.

Suppose $G_1$ is a reductive algebraic group. Fix an embedding
$\g_1\hookrightarrow \g$ such that $\t\subset \n_\g(\g_1)$. Then
$\t\cap\g_1$ is a Cartan subalgebra and $\b\cap\g_1$ is a Borel
subalgebra of $\g_1$. For fixed Borel subgroup and maximal torus in
$G_1$ we take those with the Lie algebras $\b_1,\t_1$.

{\bf Homomorphisms and representations.} All homomorphisms of
reductive algebraic Lie algebras (for instance, representations) are
assumed to be differentials of homomorphisms of the corresponding
algebraic groups.

{\bf Identification $\g\cong \g^*$}. Let $G$ be a reductive
algebraic group. There is a $G$-invariant symmetric bilinear form
$(\cdot,\cdot)$ on $\g$ such that its restriction to $\t(\R)$ is
positively definite. For instance, if $V$ is a locally effective
$G$-module, then $(\xi,\eta)=\tr_V(\xi\eta)$ has the required
properties. Note that if $H$ is a reductive subgroup of $G$, then
the restriction of $(\cdot,\cdot)$ to $\h$ is nondegenerate, so one
may identify $\h$ with $\h^*$.


{\bf Parabolic subgroups and Levi subgroups}. A parabolic subgroup
of $G$ is called  {\it antistandard} if it contains the Borel
subgroup
 $B^-$ that contains $T$ and is opposite to $B$. It is known that any parabolic subgroup is $G$-conjugate
to a unique  antistandard one. Antistandard parabolics are in
one-to-one correspondence with subsets of $\Pi(\g)$. Namely, one
assigns to $\Sigma\subset \Pi(\g)$ the  antistandard parabolic
subgroup, whose Lie algebra is generated by  by $\b^-$ and
$\g^\alpha,\alpha\in \Sigma$.

By a standard Levi subgroup in $G$ we mean the Levi subgroup
containing $T$ of  an antistandard parabolic subgroup.

{\bf Simple Lie algebras, their roots and weights.} Simple roots of
a simple Lie algebra $\g$ are denoted by $\alpha_i$. The numeration
is described below. By $\pi_i$ we denote the fundamental weight
corresponding to $\alpha_i$.

{\it Classical algebras}. In all cases for $\b$ (resp. $\t$) we take
the algebra of all upper triangular (resp., diagonal) matrices in
$\g$.

$\g=\sl_n$. Let $e_1,\ldots,e_n$ denote the standard basis in
$\C^{n}$ and $e^1,\ldots,e^n$ the dual basis in $\C^{n*}$.  Choose
the generators $\varepsilon_i, i=\overline{1,n},$ of $\t^*$ given by
$\langle\varepsilon_i,\operatorname{diag}(x_1,\ldots,x_n)\rangle=x_i$.
Put $\alpha_i=\varepsilon_i-\varepsilon_{i+1}, i=\overline{1,n-1}$.

$\g=\so_{2n+1}$. Let $e_1,\ldots,e_{2n+1}$ be the standard basis in
$\C^{2n+1}$. We suppose  $\g$ annihilates the form
$(x,y)=\sum_{i=1}^{2n+1}x_iy_{2n+2-i}$. Define
$\varepsilon_i\in\t^*, i=\overline{1,n},$ by $\langle \varepsilon_i,
\operatorname{diag}(x_1,\ldots,x_n,0,-x_n,\ldots,-x_1)\rangle$
$=x_i$. Put $\alpha_i=\varepsilon_i-\varepsilon_{i+1},
i=\overline{1,n-1}, \alpha_n=\varepsilon_n$.

$\g=\sp_{2n}$. Let $e_1,\ldots,e_{2n}$ be the standard basis in
$\C^{2n}$. We suppose that $\g$ annihilates the form
$(x,y)=\sum_{i=1}^{n}(x_iy_{2n+1-i}-y_ix_{2n+1-i})$. Let us define
$\varepsilon_i\in\t^*, i=\overline{1,n},$ by $\langle \varepsilon_i,
\operatorname{diag}(x_1,\ldots,x_n,-x_n,\ldots,-x_1)\rangle=x_i$.
Put $\alpha_i=\varepsilon_i-\varepsilon_{i+1}, i=\overline{1,n-1},
\alpha_n=2\varepsilon_n$.

$\g=\so_{2n}$. Let $e_1,\ldots,e_{2n}$ be the standard basis in
$\C^{2n}$. We suppose that $\g$ annihilates the form
$(x,y)=\sum_{i=1}^{2n}x_iy_{2n+1-i}$. Define $\varepsilon_i\in\t^*,
i=\overline{1,n},$ in the same way as for $\g=\sp_{2n}$. Put
$\alpha_i=\varepsilon_i-\varepsilon_{i+1}, i=\overline{1,n-1},
\alpha_n=\varepsilon_{n-1}+\varepsilon_n$.

{\it Exceptional algebras}. For roots and weights of exceptional Lie
algebras we use the notation from \cite{VO}. The numeration of
simple roots is also taken from \cite{VO}.

{\bf Subalgebras in semisimple Lie algebra.} For semisimple
subalgebras of exceptional Lie algebras we use the notation from
\cite{Dynkin}. Below we explain the notation for classical algebras.

Suppose $\g=\sl_n$. By $\sl_k,\so_k,\sp_{k}$ we denote the
subalgebras of $\sl_n$ annihilating a subspace  $U\subset \C^n$
 of dimension $n-k$, leaving its complement $V$ invariant, and (for
$\so_k,\sp_k$) annihilating a nondegenerate orthogonal or symplectic
form on $V$.

The subalgebras $\so_k\subset\so_n, \sp_k\subset \sp_n$ are defined
analogously. The subalgebra $\gl_k^{diag}$ is embedded into
$\so_n,\sp_n$ via the direct sum of $\tau,\tau^*$ and a trivial
representation (here $\tau$ denotes the tautological representation
of  $\gl_k$). The subalgebras
$\sl_k^{diag},\so_k^{diag},\sp_k^{diag}\subset \so_n,\sp_n$ are
defined analogously. The subalgebra $G_2$ (resp., $\spin_7$) in
$\so_n$ is the image of $G_2$ (resp., $\so_7$) under the direct sum
of the 7-dimensional irreducible (resp., spinor) and the trivial
representations.

Finally, let $\h_1,\h_2$ be subalgebras of $\g=\sl_n,\so_n,\sp_n$
described above. While writing $\h_1\oplus\h_2$, we always mean that
$(\C^n)^{\h_1}+(\C^n)^{\h_2}=\C^n$.

The description above determines a subalgebra uniquely up to
conjugacy in $\Aut(\g)$.

Now we list some notation used in the text.

\begin{longtable}{p{3.3cm} p{11.7cm}}
\\$\sim_G$& the equivalence relation induced by an action of group
$G$.
\\$A^{(B)}$& the subset of all $B$-semiinvariant functions in a $G$-algebra $A$.\\
$A^\times$& the group of all invertible elements of an algebra $A$.
\\ $\Aut(\g)$& the group of
automorphisms of a Lie algebra $\g$.
\\ $\Aut^G(X)$& the group of $G$-automorphisms of a $G$-variety $X$.
\\$e_\alpha$& a nonzero element of the root subspace $\g^\alpha$.
\\ $(G,G)$& the commutant of a group $G$.
\\ $[\g,\g]$& the commutant of a Lie algebra $\g$.
\\ $G^{\circ}$& the connected component of unit of an algebraic group $G$.
\\ $G*_HV$& the homogeneous bundle over $G/H$ with  fiber $V$.
\\ $[g,v]$& the equivalence class of $(g,v)$ in $G*_HV$.
\\ $G_x$& the stabilizer of $x\in X$ under an action
$G:X$.
\\ $\g^{\alpha}$& the root subspace of $\g$ corresponding to a root  $\alpha$.
\\ $\g^{(A)}$& the subalgebra $\g$ generated by $\g^{\alpha}$ with $\alpha\in A\cup
-A$.
\\ $G^{(A)}$& the connected subgroup of $G$ with Lie algebra
$\g^{(A)}$.
\\ $m_G(X)$&$:=\max_{x\in X}\dim Gx$.
\\ $N_G(H)$ ($N_G(\h)$)& the normalizer of a  subgroup $H$ (subalgebra $\h\subset\g$) in
a group $G$.\\ $\n_\g(\h)$& the normalizer  of a subalgebra $\h$ in
a Lie algebra $\g$.
\\ $\rank(G)$& the rank of an algebraic group $G$.
\\ $\Rad_u(H)$ ($\Rad_u(\h)$)& the unipotent radical of an
algebraic group $H$  (of an algebraic Lie algebra $\h$).
\\
 $s_\alpha$& the reflection in a Euclidian space corresponding to a vector $\alpha$.
\\ $V^\g$& $=\{v\in V| \g v=0\}$, where $\g$ is a Lie algebra and
$V$ is a $\g$-module.
\\ $V(\lambda)$& the irreducible module of the highest weight $\lambda$
over a reductive algebraic group or a reductive Lie algebra.\\
$W(\g)$& the Weyl group of a reductive Lie algebra $\g$.\\
$\X(G)$& the character lattice of an algebraic group $G$.\\
$\X_{G}$& the weight lattice of a reductive algebraic group $G$.\\
$X^G$& the fixed point set for an action $G:X$.\\
$X\quo G$& the categorical quotient for an action $G:X$, where $G$
is a reductive group and $X$ is an affine  $G$-variety.\\
$\#X$& the number of elements in a set $X$.
\\   $Z_G(H)$, ($Z_G(\h)$)& the centralizer of a subgroup $H$ (of a subalgebra $\h\subset\g$)
in an algebraic group $G$.
\\ $Z(G)$&$:=Z_G(G)$.
\\ $\z_\g(\h)$& the centralizer of a subalgebra $\h$ in $\g$.
\\ $\z(\g)$& $:=\z_\g(\g)$.
\\ $\alpha^\vee$& the dual root to $\alpha$.
\\  $\Delta(\g)$& the root system of a reductive Lie algebra $\g$.
\\ $\lambda^*$& the dual highest weight to $\lambda$.
\\ $\Lambda(\g)$& the root lattice of a reductive Lie algebra $\g$.
\\ $\Lambda_{G,X}$& the root lattice of a $G$-variety $X$.
\\ $\xi_s,\xi_n$& semisimple and nilpotent parts of an element $\xi$
in an algebraic Lie algebra.
\\ $\Pi(\g)$& the system of simple roots for a reductive Lie algebra $\g$.
\\$\pi_{G,X}$& the (categorical) quotient morphism $X\rightarrow X\quo G$.
\end{longtable}

\section{Known results and constructions}\label{SECTION_prelim}
\subsection{Introduction} This section does not contain new results.
 In the first subsection we quote
definitions and basic properties of central automorphisms and root
lattices of $G$-varieties. Our exposition is based mostly on
\cite{Knop8}. In the second subsection we show how to reduce the
computation of $\X_{G,X}$ to the case when $X$ is an affine
homogeneous space of rank $\rank G$. The reduction is based on
results of  \cite{Panyushev2},\cite{comb_ham}.
\subsection{Central automorphisms and root lattices}
Let $G,X$ be such as in Introduction.

The following definition was given in~\cite{Knop8}.
\begin{defi}\label{Def:1.3.1}
A $G$-\!\! automorphism   $\varphi$ of  $X$ is said to be {\it
central} if for any $\lambda\in \X_{G,X}$ the automorphism $\varphi$
acts on $\C(X)^{(B)}_\lambda$ by a constant. Central automorphisms
of $X$ form a group denoted by
$\mathfrak{A}_{G}(X)$\index{agx@$\mathfrak{A}_{G}(X)$}.
\end{defi}

However, the group $\mathfrak{A}_{G}(X)$ has a disadvantage   not to
be a birational invariant of $X$. This problem is fixed as follows.
By \cite{Knop8}, Theorem 5.1,  any  open $G$-subvariety $X^0\subset
X$ is stable with respect to $\mathfrak{A}_G(X)$ and there is the
inclusion $\mathfrak{A}_{G}(X)\subset \mathfrak{A}_G(X^0)$. It turns
out that there is a unique maximal group of the form
$\mathfrak{A}_{G}(X^0)$ (\cite{Knop8}, Corollary 5.4). We denote
this group by $\mathfrak{A}_{G,X}$\index{agx@$\A_{G,X}$}.

\begin{Lem}[\cite{comb_ham}, Lemma 7.17]\label{Lem:1.3.2}
If $X$ is quasiaffine, then $\A_{G,X}=\A_{G}(X)$.
\end{Lem}

Set $A_{G,X}:=\Hom(\X_{G,X},\C^\times)$. The group $\A_{G,X}$ is
embedded into $A_{G,X}$ as follows. We assign
$a_{\varphi,\lambda}\in\C^\times$ to $\varphi\in \A_{G,X},\lambda\in
\X_{G,X}$ by $\varphi f_\lambda=a_{\varphi,\lambda}f_\lambda$, $f\in
\C(X)^{(B)}_\lambda$. The map
$\iota_{G,X}\index{zzzi@$\iota_{G,X}$}:\A_{G,X}\rightarrow A_{G,X}$
is defined by $\lambda(\iota_{G,X}(\varphi))=a_{\varphi,\lambda}$.
Clearly, $\iota_{G,X}$ is a well-defined group homomorphism.

\begin{Prop}[\cite{Knop8}, Theorem 5.5]\label{Prop:1.3.3} The map $\iota_{G,X}$ is injective
and its image is closed.
\end{Prop}

The following lemma justifies the term  "central":

\begin{Lem}[\cite{Knop8}, Corollary 5.6]\label{Lem:1.3.4}
$\A_{G}(X)$ is contained in the center of $\Aut^G(X)$.
\end{Lem}


\begin{defi}\label{Def:1.4.1}
The lattice  $\Lambda_{G,X}\index{zzzl@$\Lambda_{G,X}$}\subset
\X_{G,X}=\X(A_{G,X})$ consisting of all characters annihilating
$\iota_{G,X}(\A_{G,X})$ is called the {\it root lattice} of $X$.
\end{defi}

Since the image of $Z(G)$ in $\Aut^G(X)$ lies in $ \A_{G,X}$, we see
that $\Lambda_{G,X}\subset \Lambda(\g)$.

\begin{Prop}[\cite{Knop8}, Theorem 6.3]\label{Prop:1.4.4}
Let $X_1,X_2$ be irreducible  $G$-\!\! varieties and
$\varphi:X_1\rightarrow X_2$ a dominant generically finite
$G$-morphism. Then $\Lambda_{G,X_1}=\Lambda_{G,X_2}$.
\end{Prop}

In particular, the lattice $\Lambda_{G,G/H}$ depends only on the
pair $\g,\h$, so we write
$\Lambda(\g,\h)\index{zzzl@$\Lambda(\g,\h)$}$ instead of
$\Lambda_{G,G/H}$.

\begin{Prop}\label{Prop:1.4.10}
$\Lambda_{G,X}=\Lambda_{G,Gx}$ for $x\in X$ in general position.
\end{Prop}
\begin{proof}
This follows directly from \cite{Knop8}, Theorem 5.9.
\end{proof}

\begin{Lem}\label{Prop:1.4.5}
Let $X_1,X_2$ be homogeneous $G$-spaces and $\varphi:X_1\rightarrow
X_2$ a dominant  $G$-morphism. Then $\X_{G,X_2}\subset \X_{G,X_1}$.
Suppose $X_1,X_2$ are, in addition, quasiaffine. Then
$\Lambda_{G,X_2}\subset \Lambda_{G,X_1}$ and there exists a unique
homomorphism $\A_{G,X_1}\rightarrow \Aut^G(X_2)$ such that $\varphi$
becomes
 $\A_{G,X_1}$-equivariant. Its image is contained in $\A_{G,X_2}$. The dual
 to the corresponding homomorphism $\A_{G,X_1}\rightarrow
\A_{G,X_2}$ coincides with the homomorphism
   $\X_{G,X_2}/\Lambda_{G,X_2}\rightarrow
 \X_{G,X_1}/\Lambda_{G,X_1}$ induced by the inclusions of lattices.
\end{Lem}
\begin{proof}
The claim on inclusions of the weight lattices is clear. Below
$X_1,X_2$ are quasiaffine.

Recall that $\A_{G,X_1}$ acts on $X_1$ by central $G$-\!\!
automorphisms (Lemma~\ref{Lem:1.3.2}). The subalgebra
$\C[X_2]\subset \C[X_1]$ is $G$-stable whence  $\A_{G,X_1}$-stable.
From the existence of a  $T$-embedding $\C(X_2)^{(B)}\hookrightarrow
\C(X_1)^{(B)}$ it follows that
 $\A_{G,X_1}$ acts on $X_2$ by central automorphisms.
This observation implies that the required homomorphism
$\A_{G,X_1}\rightarrow \Aut^G(X_2)$ exists and is unique.

Consider the homomorphism $\A_{G,X_1}\rightarrow \A_{G,X_2}$ we have
just constructed and the natural epimorphism
$A_{G,X_1}\twoheadrightarrow A_{G,X_2}$. It is clear that the
following diagram is commutative.

\begin{picture}(60,30)
\put(4,22){$\A_{G,X_1}$}\put(30,22){$\A_{G,X_2}$}
\put(4,2){$A_{G,X_1}$}\put(30,2){$A_{G,X_2}$}
\put(10,24){\vector(1,0){18}} \put(10,4){\vector(1,0){18}}
\put(6,20){\vector(0,-1){14}} \put(32,20){\vector(0,-1){14}}
\end{picture}

Therefore all elements of $\Lambda_{G,X_2}\subset \X_{G,X_2}
\subset\X_{G,X_1}$ are annihilated on $\A_{G,X_1}$ whence the
inclusion of root lattices. The homomorphism $\A_{G,X_1}\rightarrow
\A_{G,X_2}$ possesses the required properties.
\end{proof}

Now we present the definition of the root system of a $G$-variety.
To this end we need the following general construction.

\begin{defi}\label{Def:1.4.6}
Let $V$ be a finitely dimensional euclidian vector space, $\Lambda$
a lattice in $V$, $\Gamma$ a finite subgroup in $\O(V)$ generated by
reflections and stabilizing  $\Lambda$. By the  {\it minimal root
system} associated with $\Gamma,\Lambda$ we mean the set consisting
of all primitive  $v\in \Lambda$ such that $s_v\in\Gamma$.
\end{defi}

It is easy to see that any minimal root system is a genuine reduced
root system.

Recall that to $X$ one assigns the finite group $W_{G,X}\in
\GL(\a_{G,X})$ generated by reflections and stabilizing $\X_{G,X}$
(see, for example, \cite{Weyl}, Theorem 1.1.4). By \cite{Knop8},
Corollary 6.2, $\Lambda_{G,X}$ is $W_{G,X}$-stable.

\begin{defi}\label{Def:1.4.7}
The {\it root system} $\Delta_{G,X}$ of  $X$ is the minimal root
system associated with $W_{G,X},\Lambda_{G,X}$.
\end{defi}

The root system of $X$ has some properties analogous to those of the
root system of a reductive Lie algebra.

\begin{Prop}[\cite{Knop8}, Corollary 6.5]\label{Prop:1.4.8}
 $\Delta_{G,X}$ generates  $\Lambda_{G,X}$ and the Weyl group of $\Delta_{G,X}$
 coincides with $W_{G,X}$.
\end{Prop}

Now let us state one results regarding $\Lambda_{G,X}$ obtained in
\cite{comb_ham}.

\begin{Prop}[\cite{comb_ham}, Proposition 8.8]\label{Prop:1.47}
Suppose that $X$ is quasiaffine and $\rank_G(X)=\rank G$. Let
$G=Z(G)^\circ G_1\ldots G_k$ be the decomposition of $G$ into the
locally direct product of the unit component of the center and
simple normal subgroups. Then $\Lambda_{G,X}=\bigoplus_{i=1}^k
\Lambda_{G_i,X}$.
\end{Prop}
\subsection{Reduction of the computation of $\X_{G,X}$}
At first, let us explain the reduction of computing $\X_{G,X}$ for
homogeneous spaces $X$ to that for affine homogeneous vector
bundles.

Let $H$ be an algebraic subgroup of $G$. It is known that there a
parabolic subgroup $Q\subset G$ and its Levi subgroup $M$ such that
$\Rad_u(H)\subset \Rad_u(Q)$ and $S:=M\cap H$ is a maximal reductive
subgroup of $H$. Replacing $H$ with a conjugate subgroup, we may
assume that $Q$ is antistandard and $M$ is standard.

The following result is due to Panyushev, \cite{Panyushev2}, cf.
\cite{Weyl}, Proposition 3.2.9.

\begin{Prop}\label{Prop:1.5.4}
Let $Q,M,S,H$ be as above. Then
$\X_{G,G/H}=\X_{M,M*_{S}(\Rad_u(\q)/\Rad_u(\h))}$.
\end{Prop}

Next, we are going to reduce the computation for affine homogeneous
vector bundles to that for affine homogeneous spaces.

First of all, set
\begin{equation}\label{eq:0.4:1}L_{G,X}:=Z_G(\a_{G,X}).\end{equation}
\begin{equation}\label{eq:0.4:2}L_{0\,G,X}:=\{g\in L|\chi(g)=1, \forall \chi\in \X_{G,X}
\}.\end{equation}

\begin{defi}\label{Def:1.5.7}
Let $X$ be a smooth quasiaffine $G$-variety, $L_1$  a normal
subgroup of $L_{0\,G,X}$. There is a unique irreducible (=connected)
component $\underline{X}\subset X^{L_1}$ such that
$\overline{U\underline{X}}=X$ (see \cite{comb_ham}, Proposition
8.4). This component $\underline{X}\subset X^{L_1}$ is said to be
{\it distinguished}.
\end{defi}

In the sequel we will need to extend the definition of $L_{0\,G,X}$
to actions of certain disconnected groups $G$.

\begin{defi}\label{Def:1.5.5}
A reductive algebraic group $\widetilde{G}$  is called {\it almost
connected} if  $\widetilde{G}=\widetilde{G}^\circ Z(\widetilde{G})$.
\end{defi}

Let $\widetilde{G}$ be an almost connected group with
$\widetilde{G}^\circ=G$. Set $\widetilde{B}:=N_{\widetilde{G}}(B),
\widetilde{T}:=Z_{\widetilde{G}}(T)$. Since the group
$\widetilde{G}$ is almost connected, we see that
$\widetilde{T}\subset \widetilde{B}$ and
$\widetilde{B}=\widetilde{T}\rightthreetimes U$. So the groups
$\X(\widetilde{B})$ and $\X(\widetilde{T})$ are identified.

For an irreducible $\widetilde{G}$-variety $X$ put
\begin{equation}
\X_{\widetilde{G},X}:=\{\chi\in \X(\widetilde{B})| \exists f\in
\C(X)|  b.f=\chi(b)f, \forall b\in \widetilde{B}\}.
\end{equation}

By definition, put $\a_{\widetilde{G},X}=\a_{G,X}$. The subgroups
$L_{\widetilde{G},X}, L_{0\,\widetilde{G},X}\subset \widetilde{G}$
are defined by formulas (\ref{eq:0.4:1}),(\ref{eq:0.4:2}), resp.,
where $G$ is replaced with $\widetilde{G}$. It follows directly from
definition that $L_{0\,\widetilde{G},X}$ is  almost connected.

The following proposition is essentially due to Panyushev,
\cite{Panyushev2}, and is proved in the same way as its analogue in
\cite{Weyl}, Proposition 3.2.12.

\begin{Prop}\label{Prop:1.5.8} Let $H$ be a reductive subgroup in $G$, $V$ an $H$-module
and $\pi$ the natural projection $G*_HV\rightarrow G/H$. Put
$L_1=L_{0\,G,G/H}$. Let $x$ be a point from the distinguished
component of $(G/H)^{L_1}$. Then
$L_{0\,G,G*_HV}=L_{0\,L_1,\pi^{-1}(x)}$.
\end{Prop}

The last proposition reduces the computation of $\X_{G,G*_HV}$ to
the following problems:
\begin{enumerate}
\item To determine the lattice $\X_{G,G/H}$, equivalently, the
group $L_{0\,G,G/H}$ for all reductive subgroups $H\subset G$.
\item To find a point from the distinguished component of $(G/H)^{L_{0\,G,G/H}}$
for all reductive subgroups $H\subset G$.
\item To find the group $L_{0\,\widetilde{G},V}$ for almost
connected group $\widetilde{G}$ and a $\widetilde{G}$-module $V$.
\end{enumerate}

There is an algorithm solving the third problem. It is presented,
for instance, in~\cite{Panyushev_thesis} for connected
$\widetilde{G}$. This algorithm  can be generalized directly to the
general case. We quote this algorithm in
Section~\ref{SECTION_algorithm}.

Next, we reduce the computation of $\X_{G,X}$ and the determination
of a point in the distinguished component to the case, where $X$ is
an affine homogeneous space such that $\rank_G(G/H)=\rank(G)$.

\begin{Prop}\label{Prop:1.5.3}
Let $X$ be a smooth quasiaffine $G$-variety, $L_0:=L_{0\,G,X}$, and
$\underline{X}$ the distinguished component of $X^{L^\circ_0}$. Set
$\underline{G}=N_G(L^\circ_0,\underline{X})/L^\circ_0$. Then
$\X_{G,X}=\X_{\underline{G}^\circ,\underline{X}}$ and the
distinguished components of $X^{L_0}$ and
$\underline{X}^{L_0/L_0^\circ}$ coincide. Here $\underline{X}$ is
considered as a $\underline{G}$-variety. \end{Prop}
\begin{proof}
The equality of the weight lattices was proved in \cite{comb_ham},
Theorem 8.7. Further, $\underline{X}^{L_0/L_0^\circ}$ is a union of
components of $X^{L_0}$. Let $\underline{X}'$ be the distinguished
component of $\underline{X}^{L_0/L_0^\circ}$. Then $(U\cap
N_G(L_0^\circ,\underline{X}))\underline{X}'$ is dense in
$\underline{X}$. It follows that $U\underline{X}'$ is dense in $X$
whence $\underline{X}'$ is the distinguished component of $X^{L_0}$.
\end{proof}

Now let $X=G/H$ be an affine homogeneous space. The Cartan spaces
$\a_{G,X}$ were computed in \cite{ranks}, distinguished components
$\underline{X}$ were determined in \cite{Weyl}, Section 4. It turns
out that $\underline{X}$ is an affine homogeneous
$\underline{G}^\circ$-space. So one can reduce the computation of
$\X_{G,X}$ to the case when $X$ is an affine homogeneous space of
rank $\rank(G)$.

\section{Root lattices of Hamiltonian actions}\label{SECTION_Ham}
\subsection{Introduction}
In the first subsection we define Hamiltonian actions in the
algebraic context, give some examples, state the symplectic slice
theorem that provides a local description of an affine Hamiltonian
varieties. Finally, we define an important special class of affine
Hamiltonian varieties: conical varieties.

In the second subsection we recall some definitions and results
related to Weyl groups, weight and root lattices of Hamiltonian
varieties. For simplicity, we consider only Hamiltonian
$G$-varieties $X$ such that $m_G(X)=\dim G$. After giving all
necessary definitions we state the comparison theorem (Theorem
\ref{Thm:1.22}) that relates the Weyl group, the weight and root
lattices of an affine $G$-variety $X_0$ with those of $X=T^*X_0$.
Finally, we state several technical results to be used in Subsection
\ref{SUBSECTION_root1}. The most important among them are
Propositions \ref{Prop:4.6.5},\ref{Prop:4.6.3}. Subsections
\ref{SUBSECTION_Ham_prelim},\ref{SUBSECTION_Ham_root} do not contain
new results.

In Subsection \ref{SUBSECTION_root1} we study root lattices of
affine Hamiltonian $G$-varieties from a certain class. This class
includes all cotangent bundles $T^*X_0$, where $X_0$ is an affine
$G$-variety such that $\rank_G(X_0)=\rank G$. These results
(Propositions \ref{Prop:7.1.4},\ref{Prop:7.1.8}, Corollary
\ref{Cor:7.1.7}) play a crucial role in the proof of Theorem
\ref{Thm:7.0.2}.

\subsection{Preliminaries}\label{SUBSECTION_Ham_prelim}
Let $X$ be a symplectic variety with symplectic form $\omega$ and
$G$ a reductive algebraic group acting on $X$ by symplectomorphisms.
This action is called {\it Hamiltonian} if it is equipped with a
linear $G$-equivariant map $\g\rightarrow \C[X], \xi\mapsto H_\xi,$
such that $\{H_\xi,f\}=\xi_* f$. Here $\xi_*$ denotes the velocity
vector field associated with $\xi$. The variety $X$ is called a {\it
Hamiltonian $G$-variety}. The {\it moment map} of $X$ is the
morphism $X\rightarrow \g^*$ defined by
$\langle\mu_{G,X}(x),\xi\rangle=H_\xi(x)$. 
In the sequel we fix a $G$-invariant nondegenerate symmetric form
$(\cdot,\cdot)$ on $\g$ and identify $\g$ with $\g^*$.

We say that an irreducible Hamiltonian $G$-variety $X$ is {\it
coisotropic} if a $G$-orbit of $X$ in general position is a
coisotropic $G$-variety. If $m_G(X)=\dim G$, then $X$ is coisotropic
iff $\dim X=\dim G+\rank G$, see, for instance, \cite{fibers}, Lemma
2.22.

Let us present three examples of Hamiltonian $G$-varieties.

\begin{Ex}[Symplectic vector spaces]\label{Ex:2.1.5}  Let $V$ be a symplectic vector space and
 $G$ be a reductive group acting on $V$ by linear
symplectomorphisms. Then the action $G:V$ is Hamiltonian. The moment
map $\mu_{G,V}$ is given by  $\langle\mu_{G,V}(v), \xi\rangle =
\frac{1}{2}\omega(\xi v,v), \xi\in\g, v\in V$.
\end{Ex}

\begin{Ex}[Cotangent bundles]\label{Ex:2.1.4} Let $Y$ be a smooth $G$-variety. Let $X$ be the cotangent
bundle of $Y$. Then $X$ is a symplectic algebraic variety. The
action of $G$ on $X$ is Hamiltonian. The moment map is given by
$\langle\mu_{G,X}((y,\alpha)), \xi\rangle=\langle \alpha,
\xi_{*}y\rangle$. Here $y\in Y, \alpha\in T^*_yY,\xi\in\g$.
\end{Ex}

\begin{Ex}[Model varieties]\label{Ex:2.1.6}
This example was introduced in \cite{slice}. It generalizes Example
\ref{Ex:2.1.5} and partially Example \ref{Ex:2.1.4}.
 Let $H$  be a reductive subgroup of $G$, $\eta\in \g^H$, $V$ a symplectic $H$-module. Put $U=(\z_\g(\eta)/\h)^*$.
 There is a certain closed $G$-invariant 2-form
$\omega$ on the homogeneous vector bundle $X=G*_H(U\oplus V)$
depending on the choice of an $\sl_2$-triple $(\eta_n,h,f)$ in
$\z_\g(\eta_s)^H$, see \cite{slice} or \cite{fibers}, Example 2.5.
By the model variety $M_G(H,\eta,V)$ we mean the set of all points
of $X$, where $\omega$ is nondegenerate. It was proved in
\cite{slice} that $G/H\subset M_G(H,\eta,V)$ and $X=M_G(H,\eta,V)$
for nilpotent $\eta$. By the base point of $M_G(H,\eta,V)$ we mean
$[1,(0,0)]\in G*_{H}(U\oplus V)$.
 The moment map of $M_G(H,\eta,V)$ is
constructed as follows. Identify $U$ with
$\z_\g(\eta_s+f)\cap\h^\perp$ by means of $(\cdot,\cdot)$. Then
$$\mu_{G,M_G(H,\eta,V)}([g,(u,v)])=\Ad(g)(\eta+u+\mu_{H,V}(v)).$$
Actually, the Hamiltonian structure on $M_G(H,\eta,V)$ does not
depend on the choice of $h,f$ up to an isomorphism.

If $\eta=0,H=G$ (resp., $\eta=0, V=\{0\}$), $M_G(H,\eta,V)$ is the
symplectic vector space $V$ (resp., the cotangent bundle
$T^*(G/H)$).
\end{Ex}

Let us explain why the previous example is important. Let $X$ be an
affine Hamiltonian $G$-variety and $x$ a point in $X$ with closed
$G$-orbit. It turns out that in a small neighborhood of $x$ the
variety $X$ looks like a model variety.

To state a precise result we define some invariants of the triple
$(G,X,x)$. Put $H=G_x, \eta=\mu_{G,X}(x)$. The subgroup $H\subset G$
is reductive and $\eta\in \g^H$. Put $V=(\g_*x)^\skewperp/(\g_*x\cap
\g_*x^\skewperp)$. This is a symplectic $H$-module. We say that
$(H,\eta,V)$  is the {\it determining triple} of  $X$ at $x$. For
example, the determining triple of $X=M_G(H,\eta,V)$ in
$x=[1,(0,0)]$ is $(H,\eta,V)$, see \cite{slice}, assertion 4 of
Proposition 1.

\begin{defi}\label{defi:4.3.1}
Let $X_1,X_2$ be affine Hamiltonian $G$-varieties, $x_1\in X_1,
x_2\in X_2$ be points with closed  $G$-orbits. The pairs
$(X_1,x_1),(X_2,x_2)$ are called {\it analytically equivalent}, if
there are saturated open analytical neighborhoods $O_1,O_2$ of
$x_1\in X_1, x_2\in X_2$, respectively, and an isomorphism
$O_1\rightarrow O_2$ of complex-analytical Hamiltonian $G$-manifolds
that maps $x_1$ to $x_2$.
\end{defi}

Recall that a subset of an affine $G$-variety $X$ is said to be {\it
saturated} if it is the union of fibers of $\pi_{G,X}$.

\begin{Rem}\label{Rem:4.3.2}
An open saturated analytical neighborhood in  $X$ is the inverse
image of an {\it open} analytical neighborhood in $X\quo G$ under
$\pi_{G,X}$. See, for example, \cite{slice}, Lemma 5.
\end{Rem}

\begin{Prop}[Symplectic slice theorem, \cite{slice}]\label{Prop:4.3.3}
Let $X$ be an affine Hamiltonian  $G$-variety,  $x\in X$  a point
with closed $G$-orbit, $(H,\eta,V)$  the determining triple of $X$
at $x$, and $x'$  the base point of $M_G(H,\eta,V)$. Then the pair
$(X,x)$ is analytically equivalent to the pair
$(M_{G}(H,\eta,V),x')$.
\end{Prop}

In the sequel we will often consider Hamiltonian varieties equipped
with an action of $\C^\times$ satisfying some compatibility
conditions. Here is the precise definition.

\begin{defi}\label{defi:2.2.1}
An affine Hamiltonian $G$-variety $X$ equipped with an action
$\C^\times:X$ commuting with the action of $G$ is said to be {\it
conical}  if the following  conditions (Con1),(Con2)  are fulfilled
\begin{itemize}
\item[(Con1)] The morphism $\C^\times\times X\quo G\rightarrow X\quo
G, (t,\pi_{G,X}(x))\mapsto \pi_{G,X}(tx),$ can be extended to a
morphism $\C\times X\quo G\rightarrow X\quo G$.
\item[(Con2)] There exists a positive integer $k$ (called the {\it degree} of $X$) such that
$t.\omega=t^{k}\omega$ and $\mu_{G,X}(tx)=t^k\mu_{G,X}(x)$  for all
$t\in \C^\times, x\in X$.
\end{itemize}
\end{defi}

\begin{Ex}[Cotangent bundles]\label{Ex:2.2.2}
Let $Y,X$ be such as in Example \ref{Ex:2.1.4}. The variety $X$ is a
vector bundle over $Y$. The action $\C^\times:X$ by the fiberwise
multiplication turns $X$ into a conical variety of degree 1.
\end{Ex}

\begin{Ex}[Model varieties]\label{Ex:2.2.4}
Let $H,\eta,V$ be such as in Example \ref{Ex:2.1.6} and
$X=M_G(H,\eta,V)$. Suppose that $\eta$ is nilpotent. Here we define
an action $\C^\times:X$ turning $X$ into a conical Hamiltonian
variety of degree 2. Let $(\eta,h,f)$ be an $\sl_2$-triple in
$\g^H$. Note that $h$ is the image of a coroot under an embedding of
Lie algebras.
 In particular, there
exists a one-parameter subgroup $\gamma:\C^\times\rightarrow G$ with
$\frac{d}{dt}|_{t=0}\gamma=h$. Since $[h,\h]=0, [h,f]=-2f$, we see
that  $\gamma(t)(\h^\perp)=\h^\perp,\gamma(t)(U)=U$. Define a
morphism $\C^\times\times X\rightarrow X$ by formula
\begin{equation}\label{eq:2.4}
(t, [g,(u,v)])\mapsto [g\gamma(t),t^2\gamma(t)^{-1}u,tv], t\in
\C^\times, g\in G, u\in U, v\in V.
\end{equation}
(Con1),(Con2) were checked in \cite{fibers}, Example 2.16.
\end{Ex}

\subsection{Weyl groups and root lattices for Hamiltonian
varieties}\label{SUBSECTION_Ham_root} In this section $X$ is an
irreducible affine Hamiltonian $G$-variety with symplectic form
$\omega$ such that $m_G(X)=\dim G$. It is known that the last
condition is equivalent to $\overline{\im\mu_{G,X}}=\g$. For a Levi
subalgebra $\l\subset\g$ set $\l^{pr}:=\{\xi\in\l|
\z_\g(\xi_s)\subset\l\}$. Set $X^{pr}:=G\mu_{G,X}^{-1}(\t^{pr})$. By
Propositions 4.1, 4.4 from \cite{comb_ham}, the following claims
take place:
\begin{enumerate}\item the variety $Y:=\mu_{G,X}^{-1}(\l^{pr})$ is smooth,
\item the restriction of $\omega$ to $Y$ is nondegenerate, \item the action
$L:Y$ is Hamiltonian with moment map $\mu_{G,X}|_{Y}$,
\item the natural morphism $G*_{N_G(L)}Y\rightarrow L$ is etale.
Moreover, if $L=T$, then this morphism is an open embedding with
image $X^{pr}$. In particular, the group $N_G(T)$ permutes
transitively connected components of $Y$.
\end{enumerate}
 A component of $\mu_{G,X}^{-1}(\l^{pr})$ is said to be an $L$-cross-section of $X$. Let
us fix a $T$-cross-section $X_T$. By the Weyl group of the
Hamiltonian variety $X$ (associated with $X_T$) we mean the group
$W_{G,X}^{(X_T)}:=N_G(T,X_T)/T$ considered as a linear group acting
on $\t$.

Set $\psi_{G,X}:=\pi_{G,\g}\circ\mu_{G,X}:X\rightarrow \g\quo G$. It
turns out, see \cite{alg_hamil}, Subsection 5.2, that there is a
unique morphism $\widehat{\psi}_{G,X}:X\rightarrow \t\quo
W_{G,X}^{(X_T)}$ such that $\psi_{G,X}$ is the composition of
$\widehat{\psi}_{G,X}$ and the natural finite morphism $\t\quo
W_{G,X}^{(X_T)}\rightarrow \g\quo G$.

\begin{defi}\label{defi:1.21}
Suppose $X$ is conical. We say that $X$ is {\it untwisted} if
$W_{G,X}^{(X_T)}$ is generated by reflections and the morphism
$\widehat{\psi}_{G,X}$ is smooth in codimension 1 (that is, the
subvariety of singular points of $\widehat{\psi}_{G,X}$ has
codimension at least 2 in $X$).
\end{defi}

Proceed to the definitions of weight and root lattices of $X$. Let
$T_0$ be the inefficiency kernel for the action $T:X_T$. Thanks to
(4), $T_0$ is a discrete subgroup of $T$. By the {\it weight
lattice} of $X$ we mean the annihilator of $T_0$ in $\X(T)$, we
denote the weight lattice by $\X_{G,X}^{(X_T)}$. Finally, let us
define the root lattice of $X$. We say that a Hamiltonian (that is,
$G$-equivariant and preserving $\omega$ and $\mu_{G,X}$)
automorphism $\varphi$ of $X$ is {\it central} if $\varphi(X_T)=X_T$
and the restriction of $\varphi$ to $X_T$ coincides with the
translation by some element $t_\varphi\in T/T_0$. Central
automorphisms form a subgroup of $\Aut^G(X)$ denoted by
$\A_{G,X}^{(\cdot)}$. This subgroup does not depend on the choice of
$X_T$. The map $\A_{G,X}^{(\cdot)}\rightarrow T/T_0, \varphi\mapsto
t_\varphi$ is injective, its image $\A_{G,X}^{(X_T)}$ is closed, see
\cite{comb_ham}, Corollary 5.7. By the root lattice of $X$ (denoted
$\Lambda_{G,X}^{(X_T)}$) we mean the annihilator of
$\A_{G,X}^{(X_T)}$ in $\X(T)$.

\begin{Thm}\label{Thm:1.22}
Let $X_0$ be an irreducible smooth affine $G$-variety of rank $\rank
G$. Set $X:=T^*X_0$. Then the following conditions hold:
\begin{enumerate}
\item
$m_G(X)=\dim G$. \item  $X$ is untwisted. \item There is a
$T$-cross-section $\Sigma$ of $X$ such that
$W_{G,X}^{(\Sigma)}=W_{G,X_0}, \X_{G,X}^{(\Sigma)}=\X_{G,X_0},
\Lambda_{G,X}^{(\Sigma)}=\Lambda_{G,X_0}$.
\item $L_{0\,G,X_0}$ is the stabilizer in general position for the
action $G:X$.
\item $X$ is coisotropic iff $X_0$ is spherical, that is, $B$ has an
open orbit on $X$.
\end{enumerate}
\end{Thm}
\begin{proof}
(4) was proved in \cite{Knop1}, Korollar 8.2. Thence (1). (3)
essentially was proved by Knop in \cite{Knop3},\cite{Knop8}, see
\cite{comb_ham}, Theorem 7.8. Finally, (2) stems from the equality
of the Weyl groups in (3) and \cite{Knop2}, Corollary 7.6 (see also
\cite{fibers}, Theorem 5.7, Remark 5.11). (5) was proved in
\cite{Knop1}, Satz 7.1.
\end{proof}

\begin{Lem}[\cite{comb_ham}, Lemma 6.12]\label{Lem:2.7.1}
The lattices $\Lambda_{G,X}^{(X_T)},\X_{G,X}^{(X_T)}$ are
$W_{G,X}^{(X_T)}$-stable and $w\xi-\xi\in \Lambda_{G,X}^{(X_T)}$ for
all $w\in W_{G,X}^{(X_T)}, \xi\in \X_{G,X}^{(X_T)}$.
\end{Lem}

The following  proposition follows from \cite{fibers}, Propositions
4.1,4.3.

\begin{Prop}\label{Prop:4.6.5}
Let $X,T,X_T$ be such as above and $M$ a Levi subgroup of $G$.
Suppose $0\in\im\widehat{\psi}_{G,X}$.  Let $\xi\in \z(\m)$ be a
point in general position. Then there is a  point $x\in X$
satisfying the following conditions
\begin{itemize}
\item[(a)] $\mu_{G,X}(x)_s\in \z(\m)\cap\m^{pr}$.
\item[(b)] A unique $M$-cross-section $X_M$ of $X$
containing  $x$ contains  $X_T$ and
$\widehat{\psi}_{M,X_M}(x)=\pi_{W_{M,X_M}^{(X_T)},\t}(\xi)$.
\item[(c)] $Gx$ is closed in $X$. Set $\widehat{G}:=(M,M)$. Automatically, $\widehat{G}x$ is closed in $X_M$.
\item[(d)]  The
Hamiltonian $\widehat{G}$-variety
$\widehat{X}:=M_{\widehat{G}}(H\cap \widehat{G},\eta_n,V/V^H)$
 is coisotropic, where $(H,\eta,V)$ is the determining triple of $X_M$ (or, equivalently, of $X$,
 see \cite{fibers}, Lemma 2.27) at $x$.
\item[(e)] $H^\circ\subset \widehat{G}$.
\end{itemize}
\end{Prop}

It is easy to see that $m_{\widehat{G}}(X_M)=\dim \widehat{G}$. The
slice theorem (Proposition \ref{Prop:4.3.3}) implies
$m_{\widehat{G}}(\widehat{X})=\dim\widehat{G}$. Thus the condition
that $\widehat{X}$ is coisotropic means $\dim\widehat{X}=\dim
\widehat{G}+\rank\widehat{G}$.

\begin{Prop}\label{Prop:4.6.3}
Let $X,T,X_T,M,X_M,\widehat{G}$ be such as in
Proposition~\ref{Prop:4.6.5}, $\widehat{T}:=T\cap \widehat{G}$.
Suppose $x\in X$ satisfies conditions (a)-(e) of Proposition
\ref{Prop:4.6.5}. Let $\widehat{X}$ denote the model variety
constructed from $x$ in Proposition \ref{Prop:4.6.5}. Then there is
a $\widehat{T}$-section $\widehat{X}_{\widehat{T}}$ of $\widehat{X}$
satisfying the following conditions.
\begin{enumerate}
\item There is the inclusion
\begin{equation}\label{eq:4.6:9}
W_{\widehat{G},\widehat{X}}^{(\widehat{X}_{\widehat{T}})}\subset
(W_{G,X}^{(X_T)})\cap M/T.\end{equation} If $X$ is conical and
untwisted, then $\widehat{X}$ is also untwisted, and
(\ref{eq:4.6:9}) turns into an equality.
\item Let $p$ denote the orthogonal projection $\t\rightarrow
\widehat{\t}$. Then \begin{equation}\label{eq:4.6:10}
p(\X_{G,X}^{(X_T)})=
\X_{\widehat{G},\widehat{X}}^{(\widehat{X}_{\widehat{T}})}.
\end{equation}
\item There is the inclusion \begin{equation}\label{eq:4.6:11}
\Lambda_{\widehat{G},\widehat{X}}^{(\widehat{X}_{\widehat{T}})}\subset
\Lambda_{G,X}^{(X_T)}\cap \widehat{\t}.\end{equation}
\end{enumerate}
\end{Prop}
\begin{proof}
 $\widehat{X}_{\widehat{T}}$ was constructed the proof of
\cite{fibers}, Proposition 4.6. Recall the construction briefly.

In the notation of Proposition \ref{Prop:4.6.5} set
$\widehat{X}':=\widehat{X}\times V^H$. By Proposition
\ref{Prop:4.3.3}, there is a connected saturated neighborhood $O$ of
$x$ in $X_M$ that is equivariantly symplectomorphic  to an open
saturated neighborhood of the base point $x'$ in $\widehat{X}'$.
Denote the last neighborhood also by $O$. Moreover, we may assume
that the intersection of $O$ with any $\widehat{T}$-cross-section of
$\widehat{X}'$ is connected.  Let $\widehat{X}'_{\widehat{T}}$ be a
$\widehat{T}$-cross-section of $\widehat{X}'$ containing a connected
component of $O\cap X_T$. By Proposition 4.6 from \cite{fibers}, we
get (\ref{eq:4.6:9}). The remaining part of assertion 1 was proved
in \cite{fibers}, Proposition 5.3.

By \cite{comb_ham}, Lemma 6.10,
$\X_{M,X_M}^{(X_T)}=\X_{G,X}^{(X_T)}, \Lambda_{M,X_M}^{(X_T)}\subset
\Lambda_{G,X}^{(X_T)}$. Tautologically, the intersection of the
inefficiency kernel for the action $T:X_T$ with $\widehat{G}$
coincides with the inefficiency kernel for the action $T\cap
\widehat{G}:X_T$. So it follows directly from the definition of
$\X^{(\bullet)}_{\bullet,\bullet}$ that
$\X_{\widehat{G},X_M}^{(X_T)}=p(\X_{M,X_M}^{(X_T)})$. Further, by
\cite{comb_ham}, Lemma 6.14,
$\Lambda_{\widehat{G},X_M}^{(X_T)}=\Lambda_{M,X_M}^{(X_T)}$. Since
$O$ admits open embeddings into both $\widehat{X}',X$, we get
$\X_{\widehat{G},\widehat{X}'}^{(\widehat{X}'_{\widehat{T}})}=\X_{\widehat{G},X_M}^{(X_T)}$.
This proves assertion 2. To prove the third assertion we may (and
will) assume that $X_M=X, \widehat{G}=G$. In this case we need to
check that $\Lambda_{G,\widehat{X}'}^{(\widehat{X}'_T)}\subset
\Lambda_{G,X}^{(X_T)}$ or, equivalently,
$\A_{G,\widehat{X}'}^{(\widehat{X}'_T)}\supset \A_{G,X}^{(X_T)}$.

Choose $\varphi\in \A_{G,X}^{(\cdot)}$. Any component component of
$O\cap X_T$ is $T$-stable whence $\varphi$-stable. Therefore  $O$
and   $\widehat{X}'_{T}\cap O$ are $\varphi$-stable. It remains to
check that there is an (automatically unique) element
$\overline{\varphi}\in \A_{G,\widehat{X}'}^{(\cdot)}$ such that
$\overline{\varphi}|_{O}=\varphi|_{O}$. (\ref{eq:4.6:9}) and Lemma
5.6 from \cite{comb_ham} yield
$\A_{G,X^{pr}}^{(X_T)}\subset\A_{G,\widehat{X}'^{pr}}^{(\widehat{X}'_T)}$,
so one can find  an element $\overline{\varphi}\in
\A_{\widehat{G},\widehat{X}'^{pr}}^{(\cdot)}$ with
$\overline{\varphi}|_{O\cap \widehat{X}'_{T}}=\varphi|_{O\cap
\widehat{X}'_T}$. Therefore $\overline{\varphi}|_{O\cap
\widehat{X}'^{pr}}=\varphi|_{O\cap \widehat{X}'^{pr}}$. So the
rational mapping $\overline{\varphi}:\widehat{X}'\dashrightarrow
\widehat{X}'$ is defined in all divisors intersecting $O$. But all
components of $\widehat{X}'\setminus \widehat{X}'^{pr}$ are
$\C^\times$-stable. Therefore any of them intersects $O$. It follows
that $\overline{\varphi}$ is a morphism. Applying \cite{comb_ham},
Lemma 5.6, we complete the proof.
\end{proof}

\subsection{Some properties of root lattices of affine Hamiltonian varieties}\label{SUBSECTION_root1}
In this subsection $X$ is an untwisted conical  affine Hamiltonian
$G$-variety such that $m_G(X)=\dim G$.

We assume, in addition, that $X$ is locally orthogonalizable in the
sense of the following definition.

\begin{defi}\label{Def:7.1.1}
A smooth affine  $G$-variety $Y$ is called {\it locally
orthogonalizable} (shortly, l.o.), if for any $y\in Y$ with closed
$G$-orbit the $G_y$-module $T_yY$ is orthogonal.
\end{defi}

\begin{Rem}\label{Rem:7.1.2}
The $G_y$-module $T_yY$ is  orthogonal iff the slice module
$T_yY/\g_*y$ is. This stems easily from the fact that the
$G_y$-module $\g_*y\cong (\g_y)^\perp$ is an orthogonal submodule of
$\g$.
\end{Rem}

The following simple lemma provides some examples of l.o.
$G$-varieties.

\begin{Lem}\label{Lem:7.1.3}
\begin{enumerate}
\item Let $H$ be a reductive subgroup of $G$ and $V$ be an
$H$-module. Then the  $G$-variety $G*_HV$ is l.o. iff the $H$-module
$V$ is orthogonal.
\item Let $Y$ be a l.o.  $G$-variety and $y\in Y$
a point with closed $G$-orbit. Then $G*_{G_y}(T_yY/\g_*y)$ is l.o.
\item Let $X_0$ be a smooth affine $G$-variety. Then $T^*X_0$ is a
l.o.  $G$-variety.
\item If $X$ is a l.o. $G$-variety and
$G^0$ is a  subgroup of $G$ containing $(G,G)$, then $X$ is l.o. as
a $G^0$-variety.
\end{enumerate}
\end{Lem}
\begin{proof}
In the notation of assertion 1, note that the $H_y=G_y$-modules
$T_y(G*_HV),\g/\h\oplus V$ are isomorphic, where $y=[1,v]\in G*_HV$.
Assertion 1 follows from Remark \ref{Rem:7.1.2}. The second
assertion follows from the Luna slice theorem. In assertion 3 note
that for any $y\in T^*X_0$ with closed $G$-orbit the  $G_y$-module
$T_y(T^*X_0)$ is isomorphic to $T_{\pi(y)}X_0\oplus
T_{\pi(y)}^*X_0$, where $\pi:T^*X_0\rightarrow X_0$ is the canonical
projection. To prove assertion  4 note that $G^0x$ is closed
provided $Gx$ is.
\end{proof}

We also need the notion of a $\g$-stratum introduced in
\cite{fibers}.

\begin{defi}
A pair $(\h,V)$, where $\h$ is a reductive subalgebra of $\g$ and
$V$ is an $\h$-module, is said to be a $\g$-stratum. Two $\g$-strata
$(\h_1,V_1)$, $(\h_2,V_2)$ are called  {\it equivalent} if there
exists $g\in G$ and a linear isomorphism
$\varphi:V_1/V_1^{\h_1}\rightarrow V_2/V_2^{\h_2}$ such that
$\Ad(g)\h_1=\h_2$ and $ (Ad(g)\xi)\varphi(v_1)=\varphi(\xi v_1)$ for
all $\xi\in\h_1, v_1\in V_1/V_1^{\h_1}$.
\end{defi}

\begin{defi}\label{Def:1.4.2} Let $Y$ be a smooth affine variety and
$y\in Y$ a point with closed $G$-orbit. The pair $(\g_y,
T_yY/\g_*y)$ is called the $\g$-stratum of $y$. We say that $(\h,V)$
is a $\g$-stratum of $Y$ if $(\h,V)$ is equivalent to a $\g$-stratum
of a point of $Y$. In this case we write $(\h,V)\rightsquigarrow_\g
Y$.
\end{defi}

For $\alpha\in \Delta(\g)$ define  $\g$-strata
$S^{(\alpha)},R^{(\alpha)}\index{ra@$R^{(\alpha)}$}$ as follows:
$S^{(\alpha)}=(\g^{(\alpha)}, \C^2\oplus
\C^2)R^{(\alpha)}=(\C\alpha^\vee,V)$, where $V$ is the
two-dimensional $\C\alpha^\vee$-module, where $\alpha^\vee$ acts
with weights $\pm 1$.

All results concerning the root lattice of $X$ are based on the
following proposition.

\begin{Prop}\label{Prop:7.1.4}
Let us fix a  $T$-section $X_T$ of $X$. Let $\alpha\in \Delta(\g)$
be such that $\alpha\not\in \Lambda_{G,X}^{(X_T)}$ but $s_\alpha\in
W_{G,X}^{(X_T)}$. Then $R^{(\alpha)}\rightsquigarrow_\g X,
2\alpha\in \Lambda_{G,X}^{(X_T)}$.
\end{Prop}
\begin{proof}
Set $\m=\t+\g^{(\alpha)}$.  This is a Levi subalgebra in $\g$.
Denote by $M$ the corresponding Levi subgroup of $G$. Applying
Proposition~\ref{Prop:4.6.5} to $M$, we find a point $x\in X$
satisfying conditions (a)-(e) of this proposition. Set
$\widehat{G}=G^{(\alpha)}$ and let $\widehat{X}$ be the coisotropic
model variety defined in condition (d) of
Proposition~\ref{Prop:4.6.5}. By Proposition~\ref{Prop:4.6.3},
 $\alpha\not\in \Lambda_{\widehat{G},
\widehat{X}}^{(\cdot)}$, $\widehat{X}$ is untwisted as a Hamiltonian
$\widehat{G}$-variety and $s_\alpha\in
W_{\widehat{G},\widehat{X}}^{(\cdot)}$. Besides, by
Lemma~\ref{Lem:7.1.3}, the $\widehat{G}$-variety $\widehat{X}$ is
l.o. Our claim will follow if we check that
$\widehat{X}=T^*(\widehat{G}/F)$, where $F^\circ$ is a maximal torus
in $\widehat{G}$. Indeed, from Theorem \ref{Thm:1.22} one can easily
deduce that $2\alpha\in \Lambda_{\widehat{G},\widehat{X}}^{(\cdot)}$
and, thanks to Proposition \ref{Prop:4.6.3}, $2\alpha\in
\Lambda_{G,X}^{(X_T)}$. Below we assume that $G=\widehat{G}\cong
\SL_2,X=\widehat{X}$.

Let $H,\eta,V$ be such that $X=M_G(H,\eta,V)$. Since $X$ is
coisotropic, we see that $\dim X=\dim G+\rank G=4$. If $\eta\neq
\{0\}$, then $H$ is discrete and $V=\{0\}$. One easily verifies that
in this case $\Lambda_{G,X}^{(\cdot)}=\Lambda(\g)$. Thus $\eta=0$
and either $H^\circ$ is a maximal torus of $G$ or $H=\widehat{G}$.
It remains to consider the case $H=G$. Here $\dim V=4$. Since $V$ is
both orthogonal and symplectic, we see that $V$ is the direct sum of
two copies of the tautological $\SL_2$-module. But in this case
$W_{G,X}^{(\cdot)}=\{1\}$.
\end{proof}

Our next task is to classify coisotropic model varieties with some
special properties.

\begin{Prop}\label{Prop:7.1.5}
Suppose $G\cong \SL_3$ and $X=M_G(H,\eta,V)$, where
$\cork_G(X)=0,\eta$ is nilpotent. Then the following conditions are
equivalent:
\begin{enumerate}
\item
$\Lambda_{G,X}^{(\cdot)}\neq \Lambda(\g)$. \item $\eta=0$ and the
pair $(\h,V)$ is indicated in Table~\ref{Tbl:7.1.6}.
\end{enumerate}
Under these equivalent conditions,  $\Lambda_{G,X}^{(\cdot)}$
depends  (up to $W(\g)$-conjugacy) only on the pair $(\h,V)$.
Generators of $\Lambda_{G,X}^{(\cdot)}$ are presented in the third
column of Table~\ref{Tbl:7.1.6}.
\end{Prop}

In the second row of Table \ref{Tbl:7.1.6} $\C_{\pm \chi}$ denotes
the one-dimensional $\h$-module corresponding to a nonzero character
$\chi$ of $\h$.

\begin{longtable}{|c|c|c|}\caption{Coisotropic model varieties with small root lattices
for $G=\SL_3$}\label{Tbl:7.1.6}
\\\hline N&$(\h,V)$& $\Lambda_{G,X}^{(\cdot)}$\\\endfirsthead\hline
N&$(\g,\h,V)$&$\Lambda_{G,X}^{(\cdot)}$\\\endhead\hline 1&
$(\sl_2,0)$&$\varepsilon_1-\varepsilon_3$\\\hline 2& $(\sl_2\times
\C,\C_\chi\oplus \C_{-\chi})$&$\varepsilon_1-\varepsilon_3$
\\\hline 3&$(\so_3,0)$&$2\alpha_1,2\alpha_2$\\\hline
\end{longtable}
\begin{proof}
Suppose $\Lambda_{G,X}^{(\cdot)}\neq \Lambda(\g)$. Since
$S^{(\alpha)}\rightsquigarrow_\g X$ or
$R^{(\alpha)}\rightsquigarrow_\g X$, we see that there is $x\in\h,
x\sim_G \alpha^\vee$, where $\alpha\in\Delta(\g)$. It follows that
$\g^H$ does not contain a nilpotent element whence $\eta=0$ and $V$
is an orthogonal $H$-module. Therefore there is a $H$-module $V_0$
such that $V\cong V_0\oplus V_0^*$. As we have seen in \cite{slice},
it follows that the Hamiltonian $G$-varieties $X$ and $T^*X_0$,
where $X_0=G*_{H}V$, are isomorphic. In particular,
$\Lambda_{G,X}^{(\cdot)}\sim_{W(\g)}\Lambda_{G,X_0}$, so the former
lattice depends only on $(\h,V)$. Since $\cork_G(X)=0$, we get
$2(\dim \g-\dim \h)+2\dim V_0=\rank\g+\dim \g$, equivalently,
\begin{equation}\label{eq:7.1:1}\dim\h-\dim V_0=3.\end{equation}
At first, consider the case  $S^{(\alpha)}\rightsquigarrow_\g X$.
Proposition 5.2.1 from \cite{Weyl}  implies
$[\h,\h]\sim_G\g^{(\alpha)},V=V^{[\h,\h]}$. From (\ref{eq:7.1:1}) it
follows that  $(\h,V)$ is one of pairs 1,2 of Table~\ref{Tbl:7.1.6}.
Let us check that $\Lambda_{G,X}^{(\cdot)}$ coincides with the
lattice indicated in Table~\ref{Tbl:7.1.6}. By
Proposition~\ref{Prop:1.4.10}, if $V_0$ is a nontrivial
1-dimensional $\h$-module, then
$\Lambda_{G,G*_HV_0}=\Lambda(\g,[\h,\h])$. To determine the last
lattice we use \cite{Kramer}, Tabelle 1, and Proposition
\ref{Prop:3.7.1}.

Now consider the case $S^{(\alpha)}\not\rightsquigarrow_\g X$. In
this case  $W_{G,X}^{(\cdot)}=W(\g)$ and
$R^{(\alpha)}\rightsquigarrow_\g X$. Since $\g$ has no spherical
modules of rank 2, see \cite{Leahy}, we get $\h\neq \g$. It follows
that $\dim\h\leqslant 4$ and, in the case of equality, $\dim V_0=1$,
so we get pair N2. So $\dim\h=3,V_0=\{0\}$ whence $\h=\so_3$. By
\cite{Kramer}, Tabelle 1,
$\X_{\operatorname{PSL}_3,\operatorname{PO}_3}=2\Lambda(\g)$.
Applying Proposition~\ref{Prop:3.7.1}, we get
$\Lambda(\g,\h)=2\Lambda(\g)$.
\end{proof}

\begin{Cor}\label{Cor:7.1.7}
Let $\g$ be a simple Lie algebra with  $\rank\g> 2$. Denote by $X_T$
a $T$-cross-section of $X$. Let $\alpha,\alpha_1\in \Delta(\g)$ be
such that $s_{\alpha}\in W_{G,X}^{(X_T)}, s_{\alpha_1}\not\in
W_{G,X}^{(X_T)}$ and $\g^{(\alpha,\alpha_1)}\cong \sl_3$. Then
$\alpha\in \Lambda_{G,X}^{(X_T)}$.
\end{Cor}
\begin{proof}
Set $\m:=\t+\g^{(\alpha,\alpha_1)}$. Since $\g\not\cong G_2$,  we
see that $\m$ is a Levi subalgebra in $\g$. Let $M$ denote the
corresponding Levi subgroup. Apply Proposition~\ref{Prop:4.6.5} to
$M$. Let $x$ be a point in $X$ satisfying the conditions (a)-(e) of
this proposition, $\widehat{G}:=(M,M),\widehat{T}:=T\cap (M,M)$, and
$\widehat{X}$ be the model variety defined in condition (d). Thanks
to Proposition~\ref{Prop:4.6.3}, there is a
$\widehat{T}$-cross-section $\widehat{X}_{\widehat{T}}$ of
$\widehat{X}$ such that
\begin{equation}\label{eq:7.1:2}
W_{\widehat{G},\widehat{X}}^{(\widehat{X}_{\widehat{T}})}=
W_{G,X}^{(X_T)}\cap M/T,\end{equation}
\begin{equation}\label{eq:7.1:3}\Lambda_{\widehat{G},\widehat{X}}^{(\widehat{X}_{\widehat{T}})}\subset
 \Lambda_{G,X}^{(X_T)}.
\end{equation}
(\ref{eq:7.1:2}) and \cite{fibers}, Corollary 4.16, imply that
$W_{\widehat{G},\widehat{X}}^{(\widehat{X}_{\widehat{T}})}$ is
generated by $s_\alpha$. By (\ref{eq:7.1:3}),
$\Lambda_{\widehat{G},\widehat{X}}^{(\cdot)}\neq
\Lambda(\widehat{\g})$. From Proposition \ref{Prop:7.1.5} it follows
that
$\Lambda_{\widehat{G},\widehat{X}}^{(\widehat{X}_{\widehat{T}})}$ is
spanned by $\alpha$. To complete the proof apply (\ref{eq:7.1:3})
one more time.
\end{proof}

Let us introduce one more $\g$-stratum. Let
$\alpha_1,\alpha_2\in\Delta(\g)$ be such that
$\g^{(\alpha_1,\alpha_2)}\cong\sl_3$. By
$\widetilde{R}^{(\alpha_1,\alpha_2)}\index{raa@$\widetilde{R}^{(\alpha_1,\alpha_2)}$}$
we denote the  $\g$-stratum $(\s,V)$, where
$\s=\so_3\subset\sl_3\cong \g^{(\alpha_1,\alpha_2)}$ and $V$ is the
5-dimensional irreducible $\s$-module.

\begin{Prop}\label{Prop:7.1.8}
Let $\g,X_T$ be such as in Corollary \ref{Cor:7.1.7} and
$\alpha_1,\alpha_2$ such as in the definition of
$\widetilde{R}^{(\alpha_1,\alpha_2)}$. Suppose $s_{\alpha}\in
W_{G,X}^{(X_T)}$ for all $\alpha\in W(\g)\alpha_1$. Then the
following conditions are equivalent:
\begin{enumerate}
\item $\widetilde{R}^{(\alpha_1,\alpha_2)}\rightsquigarrow_\g X$.
\item $\alpha_1\not\in \Lambda_{G,X}^{(X_T)}, 2\alpha_1\in \Lambda_{G,X}^{(X_T)}$.
\end{enumerate}
\end{Prop}
\begin{proof}
$(2)\Rightarrow (1)$.  Analogously to the proof of Corollary
\ref{Cor:7.1.7}, we reduce the proof to the situation $G=\SL_3,
X=M_G(H,\eta,V), \cork_G(X)=0$, $\eta$ is nilpotent. We have
$W_{G,X}^{(\cdot)}=W(\g)$. Since $\Lambda_{G,X}^{(X_T)}$ is
$W_{G,X}^{(X_T)}$-stable (Lemma~\ref{Lem:2.7.1}), we see that all
three inclusions $\alpha_1\in \Lambda_{G,X}^{(X_T)},\alpha_2\in
\Lambda_{G,X}^{(X_T)}, \alpha_1+\alpha_2\in \Lambda_{G,X}^{(X_T)}$
are equivalent. (1) follows now from Proposition \ref{Prop:7.1.5}.

$(1)\Rightarrow (2)$. Assume that (2) does not hold. Note that
$s_\alpha\in W_{G,X}^{(X_T)}, \alpha\in \Lambda_{G,X}^{(X_T)}$ for
all $\alpha\in \Delta(\g)$ of the same length with $\alpha_1$.
Indeed, by the choice $\g$ and $\alpha_1$, there are
$\alpha^1,\ldots,\alpha^k\in W(\g)\alpha_1$ such that
$s_{\alpha^k}\ldots s_{\alpha^1}\alpha=\alpha_1$. But
$s_{\alpha^j}\in W_{G,X}^{(X_T)},j=\overline{1,k}$ and
$\Lambda_{G,X}^{(X_T)}$ is $W_{G,X}^{(X_T)}$-stable.

Let $Y\subset X$ be the  set of all points $x\in X$  such that $Gx$
is closed and the $\g$-stratum of $x$ is equivalent to
$\widetilde{R}^{(\alpha_1,\alpha_2)}$. Thanks to the Luna slice
theorem, $Y$ is a locally closed subvariety of $X$ and
$\overline{Y}\quo G=\overline{Y\quo G}$ is a subvariety of pure
codimension 2 in $X\quo G$.

There is a point $x\in Y$ such that $\g_x\subset
\g^{(\alpha_1,\alpha_2)}$. Let us check that the subspace
$\ker\alpha_1\cap\ker\alpha_2\subset \t$ is a Cartan subalgebra in
$\z_\g(\g_x)$. Indeed, the inclusion
$\ker\alpha_1\cap\ker\alpha_2\subset \z_\g(\g_x)$ stems from
$\g_x\subset \g^{(\alpha_1,\alpha_2)}$. On the other hand, $\rank
\z_\g(\g_x)<\rank \g-1$, because   $\g_x\not\sim_G\g^{(\beta)}$ for
any $\beta\in\Delta(\g)$.

By the previous paragraph,
$\psi_{G,X}(Y)\subset\pi_{W(\g),\t}(\ker\alpha_1\cap\ker\alpha_2)$.
By \cite{alg_hamil}, Theorem 1.2.3, $(\psi_{G,X}\quo
G)^{-1}(\pi_{W,\t}(\ker\alpha_1\cap\ker\alpha_2))$ has pure
codimension 2 in $X\quo G$ and $\psi_{G,X}(Y)$ is dense in
$\pi_{W(\g),\t}(\ker\alpha_1\cap\ker\alpha_2)$.

Replacing $x$ with $gx$ for appropriate $g\in G$, we get
$\mu_{G,X}(x)_s\in\z(\m)\cap\m^{pr}$ (here automatically
$\g_x\subset \g^{(\alpha_1,\alpha_2)}$). Then
$\widehat{\psi}_{G,X}(x)\in
\pi_{W_{G,X}^{(\cdot)},\a_{G,X}^{(\cdot)}}(\z(\m_0)\cap\m_0^{pr})$
for some  Levi subalgebra  $\m_0\subset\g$ such that   $\t\subset
\m_0,\m_0\sim_G\m$. So, replacing  $\m$ with $\m_0$ if necessary, we
may assume that $\widehat{\psi}_{G,X}(x)\in
\pi_{W_{G,X}^{(\cdot)},\a_{G,X}^{(\cdot)}}(\z(\m)\cap\m^{pr})$.

Now let $X_T'$ be a $T$-cross-section of $X$. There exists $w\in W$
such that $wW_{G,X}^{(X_T)}w^{-1}=W_{G,X}^{(X_T')},
w\Lambda_{G,X}^{(X_T)}=\Lambda_{G,X}^{(X_T')}$. By the first
paragraph of the proof of $(1)\Rightarrow(2)$,
$s_{\alpha_1},s_{\alpha_2}\in W_{G,X}^{(X_T')},$ $\alpha_1\in
\Lambda_{G,X}^{(X_T')}$. Replacing $X_T$ with $X_T'$, if necessary,
we may assume that $X_T$ is contained in a unique $M$-section of $X$
containing $x$. So $x$ satisfies conditions (a)-(e) of Proposition
\ref{Prop:4.6.5}. Set $\widehat{G}:=(M,M)$ and let $\widehat{X}$ be
the model variety constructed in (d). By the choice of $x$, we get
$\widehat{X}=T^*(\SL_3/H)$ with $\h=\so_3$. So
$W_{\widehat{G},\widehat{X}}^{(\widehat{X}_{\widehat{T}})}=W(\widehat{\g})$,
$\Lambda_{\widehat{G},\widehat{X}}^{(\widehat{X}_{\widehat{T}})}=2\Lambda(\widehat{\g})$
(Proposition~\ref{Prop:7.1.5}).

Applying assertion 2 of Proposition~\ref{Prop:4.6.3}, we see that
for any $w\in W(\g)$ the projection of $w\alpha_1$ to
$\widehat{\t}:=\t\cap [\m,\m]$ lies in
$\X_{\widehat{G},\widehat{X}}^{(\widehat{X}_{\widehat{T}})}$ for an
appropriate  $\widehat{T}$-cross-section $\widehat{X}_{\widehat{T}}$
of $\widehat{X}$. But such projections span the lattice
$\X_{\SL_3}$. By Lemma~\ref{Lem:2.7.1},
$w\xi-\xi\in\Lambda_{\widehat{G},\widehat{X}}^{(\widehat{X}_{\widehat{T}})}$
for any $\xi\in \Lambda$, $w\in
W_{\widehat{G},\widehat{X}}^{(\widehat{X}_{\widehat{T}})}$.  Note
that $s_{\alpha_1}(\pi_1)-\pi_1=-\alpha_1$. Contradiction with
$\Lambda_{\widehat{G},\widehat{X}}^{(\widehat{X}_{\widehat{T}})}=2\Lambda(\widehat{\g})$.
\end{proof}
\section{Computation of root and weight lattices for affine homogeneous spaces}\label{SECTION_root}
\subsection{Introduction}\label{SUBSECTION_root_intro}
In this section $G$ is a connected reductive group,  $B\subset G$ is
its Borel subgroup and $T\subset B$ is its maximal torus. Throughout
the section $X=G/H$ is an affine homogeneous space  of rank
$\rank(G)$. Our objective is to compute the lattices
$\Lambda_{G,X},\X_{G,X}$. At first, we compute the root lattices and
then, using this computation, determine the weight lattices.

Recall that the root lattice $\Lambda_{G,G/H}$ depends only on the
pair $(\g,\h)$ (by Proposition~\ref{Prop:1.4.4}) so we write
$\Lambda(\g,\h)$ instead of $\Lambda_{G,G/H}$.
Lemma~\ref{Prop:1.4.5} implies that $\Lambda(\g,\h)\subset
\Lambda(\g,\h_1)$ for any ideal $\h_1\subset\h$.

Now let $\g=\z(\g)\oplus \bigoplus_{i=1}^k \g_i$ be the
decomposition into the direct sum of the center and simple ideals,
and $\h_i=\h\cap\g_i$. By Proposition \ref{Prop:1.47},
$\Lambda(\g,\h)=\bigoplus_{i=1}^k \Lambda(\g_i,\h_i)$. So it is
enough to compute $\Lambda(\g,\h)$ for simple $\g$.

\begin{defi}\label{Def:7.0.1}
Let $\g$ be simple. A reductive subalgebra  $\h\subset\g$ is called
$\Lambda$-essential if $\a(\g,\h)=\t$ and for any ideal $\h_1\subset
\h$ the inclusion $\Lambda(\g,\h)\subset \Lambda(\g,\h_1)$ is
strict.
\end{defi}

The following theorem is the main result on the computation of
$\Lambda(\g,\h)$.

\begin{Thm}\label{Thm:7.0.2}
\begin{enumerate}
\item All $\Lambda$-essential subalgebras $\h\subset\g$ together with
the corresponding lattices $\Lambda(\g,\h)$ are presented
in~\ref{Tbl:7.0.3}.
\item For any reductive subalgebra $\h\subset\g$ there is a unique ideal $\h^{\Lambda-ess}\subset\h
\index{hless@$\h^{\Lambda-ess}$}$ such that $\h^{\Lambda-ess}$ is
$\Lambda$-essential and
$\Lambda(\g,\h)=\Lambda(\g,\h^{\Lambda-ess})$. Such
$\h^{\Lambda-ess}$ is maximal (w.r.t inclusion) among all ideals in
$\h$ that are $\Lambda$-essential subalgebras in $\g$.
\end{enumerate}
\end{Thm}

\begin{longtable}{|c|c|c|c|}
\caption{$\Lambda$-essential subalgebras $\h\subset\g$ and lattices
$\Lambda(\g,\h)$}\label{Tbl:7.0.3}\\\hline
N&$\g$&$\h$&$\Lambda(\g,\h)$\\\endfirsthead\hline
N&$\g$&$\h$&$\Lambda(\g,\h)$\\\endhead\hline 1&$\sl_n,n\geqslant
2$&$\so_n$&$2\Lambda(\g)$\\\hline
2&$\sl_{2n+1}$&$\sl_{n+1}$&$\{\sum_{i\neq n+1}x_i\varepsilon_i|
x_i\in\Z, \sum_{i\neq n+1}x_i=0\}$\\\hline
3&$\sl_{2n+1}$&$\sp_{2n}$&$\{\sum x_i\varepsilon_i| x_i\in\Z,
\sum_{i}x_{2i}=\sum_{i}x_{2i+1}=0\}$\\\hline
4&$\so_{2n+1},n\geqslant 3$&$\so_{n+1}$&$\{\sum_{i=1}^n
x_i\varepsilon_i| x_i\in\Z, \sum_{i=1}^n x_i\equiv 0 (\mod 2)\}$
\\\hline 5&$\so_{2n+1},n\geqslant 3$&$\so_{n+1}\oplus \so_n$&$2\Lambda(\g)$\\\hline
6&$\so_{2n+1},n\geqslant 3$&$\gl_n$&$\{\sum_{i=1}^n
x_i\varepsilon_i| x_i\in\Z, \sum_{i=1}^{[n/2]} x_{n-2i}\equiv 0
(\mod 2)\}$\\\hline 7&$\sp_{2n},n\geqslant
2$&$\sl_n$&$\{2\sum_{i=1}^n x_i\varepsilon_i| x_i\in\Z\}$\\\hline
8&$\sp_{2n},n\geqslant 2$&$\gl_n$&$2\Lambda(\g)$\\\hline
9&$\so_{2n},n\geqslant 4$&$\so_n\oplus\so_n$&$2\Lambda(\g)$\\\hline
10&$G_2$&$A_1\times \widetilde{A}_1$&$2\Lambda(\g)$\\\hline
11&$F_4$&$C_3$&$\{\sum_{i=1}^4 x_i\varepsilon_i|
x_i\in\Z,\sum_{i=1}^4 x_i\equiv 0 (\mod 2)\}$\\\hline
12&$F_4$&$C_3\times A_1$&$2\Lambda(\g)$\\\hline
13&$E_6$&$C_4$&$2\Lambda(\g)$\\\hline
14&$E_7$&$A_7$&$2\Lambda(\g)$\\\hline
15&$E_8$&$D_8$&$2\Lambda(\g)$\\\hline
\end{longtable}

Proceed to computing  weight lattices. Till the end of the
subsection $G$ is an arbitrary connected reductive group. By
$\g_1,\ldots,\g_k$ we denote all simple ideals of $\g$.

Let us introduce some notation. Let $H$ be a reductive subgroup of
$G$ with $\rank_G(G/H)=\rank G$. Let $\widehat{H}$ denote the
connected subgroup in $G$ with Lie algebra
$\bigoplus_{i=1}^k(\h\cap\g_i)^{\Lambda-ess}$. By
$H^{\X-sat}\index{hxsat@$H^{\X-sat}$}$ we denote the inverse image
of $\A_{G,G/\widehat{H}}\subset N_G(\widehat{H})/\widehat{H}$ in
$N_G(\widehat{H})$. It follows directly from the definition that
$\X_{G,G/H^{\X-sat}}=\Lambda(\g,\h)$. Here is the main result
concerning the computation of weight lattices.

\begin{Thm}\label{Thm:7.0.5}
Let $H$ be as in the previous paragraph. Suppose $G$ is
algebraically simply connected (i.e., is the direct product of a
torus and  a simply connected semisimple group). Set $H_0:=H\cap
H^{\X-sat}$.
\begin{enumerate}
\item $\X_{G,G/H}=\X_{G,G/H_0}$. Further, $\h_0=\h^{\Lambda-ess}$ and $H_0$ is the maximal
normal subgroup of $H$ contained in $H^{\X-sat}$. Finally,
$H_0^{\X-sat}=H^{\X-sat}$.
\item Suppose
 $H\subset H^{\X-sat}$. Recall that there is the natural duality $\X_{G,G/H^\circ}/\Lambda(\g,\h)$ и
$H^{\X-sat}/H^\circ=\A_{G,G/H^\circ}$. When $G$ is simple this
duality is described in Remark \ref{Rem:7.0.6} below. In the general
case there are the equalities
\begin{equation}\label{eq:7.0:1}
\begin{split}&\X_{G,G/H^\circ}/\Lambda(\g,\h)=
\X_{Z(G)^\circ}\oplus\bigoplus_{i=1}^k
\X_{G_i,G_i/H_i^\circ}/\Lambda(\g_i,\h_i),\\
&H^{\X-sat}/H^\circ=Z(G)^\circ\times \prod_{i=1}^k
H_i^{\X-sat}/H_i^\circ.
\end{split}
\end{equation}
The duality between $\X_{G,G/H^\circ}/\Lambda(\g,\h)$ и
$H^{\X-sat}/H^\circ$ is the direct product of the dualities between
the corresponding factors in (\ref{eq:7.0:1}).
\item If $H\subset H^{\X-sat}$, then
$\X_{G,G/H}$ coincides with the inverse image of the annihilator of
$H/H^\circ$ in $\X_{G,G/H^\circ}/\Lambda(\g,\h)$ under the natural
epimorphism
$\X_{G,G/H^\circ}\twoheadrightarrow\X_{G,G/H^\circ}/\Lambda(\g,\h)$.
\end{enumerate}
\end{Thm}

\begin{Rem}\label{Rem:7.0.6}
This remark gives a more or less explicit description of the duality
between $\X_{G,G/H^\circ}/\X_{G,G/H^{\X-sat}}$ and
$H^{\X-sat}/H^\circ$, where $G$ is a simple group and  $\h$ is a
subalgebra of  $\g$ indicated in Table~\ref{Tbl:7.0.3}. By
$\chi_\lambda$ we denote the character of $H^{\X-sat}/H^\circ$
corresponding to a weight $\lambda\in \X_{G,G/H^\circ}$. Note, at
first, that the center of $G$ is identified with
$(\X_G/\Lambda(\g))^*$. Further, there is the natural homomorphism
$\X_{G,G/H^\circ}/\Lambda(\g,\h)\rightarrow \X_G/\Lambda(\g)$. So
any element of $Z(G)$ determines an element in
$(\X_{G,G/H^\circ}/\Lambda(\g,\h))^*$. Note that $Z(G)\subset
H^{\X-sat}$. Below we will see that the image of the map
$Z(G)\rightarrow (\X_{G,G/H^\circ}/\Lambda(\g,\h))^*$ equals
$Z(G)\cap H^\circ$, and the corresponding map
$\X_{G,G/H^\circ}/\Lambda(\g,\h)\rightarrow \X(Z(G)/Z(G)\cap
H^\circ)$ coincides with $\lambda\mapsto \chi_\lambda$. Therefore it
is enough to determine: \begin{enumerate} \item the lattice
$\X_{G,G/H^\circ}$, \item elements from $H^{\X-sat}$ whose images in
$H^{\X-sat}/Z(G)H^\circ$ generate the last group, \item characters
$\chi_\lambda$ for these elements $\lambda$.
\end{enumerate}
There are no  elements as in (2) precisely for the pairs $(\g,\h)$
NN10-13,15. In all remaining cases the indicated information is
presented in Table~\ref{Tbl:7.0.7}. In the first column the number
of the pair in Table~\ref{Tbl:7.0.3} is given. In the second column
we indicate an element $\lambda\in \X_{G,G/H^\circ}$ whose image in
$\X_{G,G/H^\circ}/\Lambda(\g,\h)$ generates this group. In the
square brackets the order of the image is given. Column 3 contains
the group $H^{\X-sat}/H^\circ$. If this group is finite, then we
indicate a generator of $H^{\X-sat}$. If the group is infinite (the
pairs NN2,3), then we give a typical element of a subgroup
complementing $H^\circ$ in $H^{\X-sat}$. Finally, in the last column
we indicate the character $\chi_\lambda$ for the element $\lambda$
presented in column 2.

When $\rank H=\rank G$ (whence $Z(G)\subset H$ for any $G$) we
consider the most convenient for us group $G$. In row 4 we assume
that  $G=\SO_{2n+1}$, since   $Z(G)\subset H^\circ$ for  simply
connected $G$ too. In all remaining cases we assume that $G$ is
simply connected.
\end{Rem}

\begin{longtable}{|c|c|c|c|}
\caption{Correspondence between $\X(H^{\X-sat}/H^\circ)$ and
$\X_{G,G/H^\circ}/\Lambda(\g,\h)$}\label{Tbl:7.0.7}\\\hline
N&$\lambda$&$h$&$\chi_\lambda(h)$\\\endfirsthead\hline
N&$\lambda$&$h$&$\chi_\lambda(h)$\\\endhead\hline
1&$2\pi_1[n]$&$diag(e^{i\pi/n},\ldots,e^{i\pi/n})d, d\in
\O(n)\setminus \SO(n)$&$e^{-2\pi i/n}$\\\hline
2&$\pi_1[\infty]$&$diag(t^{n+1},\ldots,t^{n+1},t^{-n},\ldots,t^{-n})$&$t^{-n-1}$\\\hline
3&$\pi_2[\infty]$&$diag(t^{2n},t^{-1},\ldots,t^{-1})$&$t^2$\\\hline
4&$\pi_1[2]$&$diag(-1,\ldots,-1,d),d\in \O(n+1),\det(d)=(-1)^n$&$-1$
\\\hline 5&$2\pi_n[2]$&$h\in N_G(\h)\setminus H^\circ$&$-1$\\\hline
6&$\pi_2[2]$&$h\in N_G(\h)\setminus H^\circ$&$-1$\\\hline
7&$\pi_n[2]$&$\exp(\pi i\pi_n/2^m), 2^m|n, 2^{m+1}\not|
n$&$-1$\\\hline 8&$2\pi_1[2]$&$h\in N_G(\h)\setminus
H^\circ$&$-1$\\\hline 9&$2\pi_n[2]$&$h\in N_G(\h)\setminus
H^\circ$&$-1$\\\hline 14&$2\pi_1$&$h\in N_G(\h)\setminus
H^\circ$&$-1$\\\hline
\end{longtable}

Let us describe the structure of this section. In
Subsection~\ref{SUBSECTION_Weyl7} we establish the equality of the
root lattice $\Lambda(\g,\h)$ and the weight lattice
$\X_{G,G/\widetilde{H}}$ for a certain subgroup
$\widetilde{H}\subset G$ constructed from $H$. In
Subsection~\ref{SUBSECTION_root1} we get some results on the
structure of the root lattices for a certain class of affine
Hamiltonian varieties. Subsections
\ref{SUBSECTION_root2},\ref{SUBSECTION_root4} are devoted to the
proofs of Theorems \ref{Thm:7.0.2},\ref{Thm:7.0.5}. The former is
based mostly on results of Subsection \ref{SUBSECTION_root1}, the
latter is quite easy. Finally in Subsection \ref{SUBSECTION_root5}
we show how to find a point from the distinguished component of
$(G/H)^{L_{0\,G,G/H}}$.

\subsection{Connection between root and weight lattices of homogeneous spaces}\label{SUBSECTION_Weyl7}

In this subsection $G$ is a connected reductive group and  $H$ is
its algebraic subgroup. Our goal in this subsection is to prove that
$\Lambda_{G,G/H}$ coincides with $\X_{G,G/\widetilde{H}}$, where
$\widetilde{H}$ is a subgroup of $N_G(H)$ constructed from $H$.

The basic idea of the construction of $\widetilde{H}$ is that
$\widetilde{H}/H\subset N_G(H)/H\cong \Aut^G(G/H)$ should contain
all central automorphisms. Namely let $Z$ denote the semisimple part
of the center of $N_G(H)/H$. For $\widetilde{H}$ we take the inverse
image of $Z$ under the canonical epimorphism
$N_G(H)\twoheadrightarrow N_G(H)/H$.

\begin{Prop}\label{Prop:3.7.1}
$\Lambda_{G,G/H}=\X_{G,G/\widetilde{H}}$.
\end{Prop}
\begin{Lem}\label{Lem:3.7.2}
Let $X_0$ be a $G$-variety and $T_0\subset\Aut^G(X_0)$ a quasitorus.
Further, let $X_1$ be a rational quotient for the action $T_0:X_0$
equipped with an action of $G$ such that the rational quotient
mapping $X\rightarrow X_0$ is $G$-equivariant. Then
$\X_{G,X_1}\subset \X(T)$ coincides with the annihilator of
$\iota_{G,X_0}(T_0\cap \A_{G,X_0})\subset A_{G,X_0}$ in
$\X_{G,X_0}\cong \X(A_{G,X_0})$.
\end{Lem}
\begin{proof}
Let us reduce the proof to the case when  $G$ is a torus. A standard
argument, compare with the proof of Theorem 1.3 in \cite{Knop5},
shows that there are open $B$-stable quasiaffine subvarieties
$X_0'\subset X_0, X_1'\subset X_1$. Embed $X_0',X_1'$ to affine
$B$-varieties $\overline{X}'_0,\overline{X}'_1$ (this is possible by
\cite{VP}, Theorem 1.4) and set $Z_i:=\Spec(\C[\overline{X}'_i]^U),
i=0,1$. Then $Z_i,i=1,2,$ is a rational quotient for the action
$U:X_i$. Clearly,  $\X_{G,X_0}=\X_{T,Z_0}, \X_{G,X_1}=\X_{T,Z_1}$.
Since $\C(Z_1)\cong \C(X_0)^{U\times T_0}$, we see that $Z_1$ is a
rational quotient for the action $T_0:Z_0$. The action
$T_0:\C(X_0)^U$ is effective, for the action $T_0:\C(X)$ is. It
follows that the action $T_0:Z_0$ is effective. For any $\lambda\in
\X_{G,X_0}=\X_{T,Z_0}$ there is a  $T_0$-isomorphism
$\C(X_0)^{(B)}_{\lambda}\cong \C(Z_0)^{(T)}_{\lambda}$. Thus
$T_0\cap\A_{G,X_0}=T_0\cap \A_{T,Z_0}$ and
$\iota_{G,X_0}|_{T_0\cap\A_{G,X_0}}=\iota_{T,Z_0}|_{T_0\cap
\A_{T,Z_0}}$. So it is enough to prove the claim of the lemma  for
the pair $(T,Z_0)$ instead of $(G,X_0)$. Further, we easily reduce
to the case when the action $T:Z_0$ is effective.

Let us note that $L_{T,Z_1},L_{T\times T_0,Z_0}$ coincide with the
inefficiency kernels of the corresponding actions whence
\begin{equation}\label{eq:Lem:3.7.2:1}
L_{T,Z_1}=\{t\in T| tz\in T_0z\text{ for }z\in Z_0\text{ in general
position}\},
\end{equation}
and
\begin{equation}\label{eq:Lem:3.7.2:2}
L_{T\times T_0,Z_0}=(T\times T_0)_z.
\end{equation}
for  $z\in Z_0$ in general position. From (\ref{eq:Lem:3.7.2:1}) and
(\ref{eq:Lem:3.7.2:2}) it follows that $L_{T,Z_1}=\pi_1(L_{T\times
T_0,Z_0})$, where $\pi_1:T\times T_0\rightarrow T$ is the projection
to the first factor. On the other hand, the action $T_0:Z_0$ is
effective whence the restriction of the projection  $\pi_2:T\times
T_0\rightarrow T_0$ to $L_{T\times T_0,Z_0}$ is an embedding. The
image of this embedding coincides with $\A_{T,Z_0}\cap T_0$, for it
consists precisely of those elements of $T_0$ that act on $Z_0$ as
elements of $T$. The homomorphism
$\pi_1\circ\pi_2^{-1}|_{T_0\cap\A_{T,Z_0}}$ maps an element $t_0\in
T_0\cap \A_{T,Z_0}$ to the element $t\in T$ such that $t_0 z=tz$ for
all $z\in Z_0$. In other words,
$\pi_1\circ\pi_2^{-1}|_{T_0\cap\A_{T,Z_0}}=\iota_{T,Z_0}|_{T_0\cap\A_{T,Z_0}}$.
Equivalently, $L_{T,Z_1}=\iota_{T,Z_0}(T_0\cap\A_{T,Z_0})$.
\end{proof}

\begin{proof}[Proof of Proposition \ref{Prop:3.7.1}]
Apply Lemma \ref{Lem:3.7.2} to $X_0:=G/H,T_0:=Z_0,
X_1=G/\widehat{H}$.
\end{proof}

\subsection{Proof of Theorem
\ref{Thm:7.0.2}}\label{SUBSECTION_root2} In this subsection $\g$ is
supposed to be simple.  Let $\Delta(\g)^{min}$ denote the subsets of
$\Delta(\g)$ consisting of all roots of minimal length and set
$\Delta(\g)^{max}:=\Delta(\g)\setminus \Delta(\g)^{max}$.

\begin{Lem}\label{Lem:7.2.1}
Let $\h$ be a nonzero $\Lambda$-essential subalgebra of $\g$. Then
\begin{enumerate}
\item $\Delta(\g)^{min}\not\subset \Lambda(\g,\h)$.
\item If $\alpha\in \Delta(\g)\setminus \Lambda(\g,\h)$, then there is $h\in \h$ such that
 $h\sim_G \alpha^\vee$ and
\begin{equation}\label{eq:7.2:1}
\tr_\g h^2=2\tr_\h h^2+8.
\end{equation}
\end{enumerate}
\end{Lem}
\begin{proof}
Assertion 1 stems from $\Span_{\Z}(\Delta(\g)^{min})=\Lambda(\g)$.
Proceed to assertion 2. By Proposition~\ref{Prop:7.1.4},
$R^{(\alpha)}\rightsquigarrow_\g T^*(G/H)$. So there is $h\in \h,
h\sim_G\alpha^\vee$ such that $(\C h, U)\rightsquigarrow_\h \g/\h$,
where $U$ has a basis  $e_1,e_2$ with $he_1=2e_1, he_2=-2e_2$. The
$\C h$-modules $\g/\h$ and $(\h/\C h)\oplus U$ differ by a trivial
summand whence (\ref{eq:7.2:1}).
\end{proof}

\begin{Lem}\label{Lem:7.2.3}
Let $\h$ be a nonzero $\Lambda$-essential subalgebra of $\g$ such
that $W(\g,\h)=W(\g)$.
\begin{enumerate}
\item If $\g$ is of types $A,D,E,G$, then $\Lambda(\g,\h)=2\Lambda(\g)$.
\item If $\g$ is of types $B,C,F$, then $\Lambda(\g,\h)=\Span_\Z(2\Delta(\g)^{min}\cup \Delta(\g)^{max})$
or $2\Lambda(\g)$.
\end{enumerate}
\end{Lem}
\begin{proof}
Recall that there is a basis of $\Lambda(\g,\h)$ that is a root
system with Weyl group  $W(\g,\h)$ (Proposition~\ref{Prop:1.4.8}).
Now the proof follows from Proposition \ref{Prop:7.1.4}.
\end{proof}

Now we recall the definition of the Dynkin index (\cite{Dynkin}).
Let $\h$ be a simple subalgebra of $\g$. We fix an invariant
non-degenerate symmetric bilinear form $K_\g$ on $\g$ such that
$K_\g(\alpha^{\vee},\alpha^{\vee})=2$ for a root
$\alpha\in\Delta(\g)$ of the maximal length. Analogously define a
form $K_\h$ on $\h$. The  {\it Dynkin index} of the embedding
$\iota:\h\hookrightarrow \g$  is, by definition,
$K_\g(\iota(x),\iota(x))/K_\h(x,x)$ (the last fraction does not
depend on the choice of $x\in\h$ such that $K_\h(x,x)\neq 0$). For
brevity, we denote the Dynkin index of $\iota$ by $i(\h,\g)$. It
turns out that $i(\h,\g)$ is a positive integer (see~\cite{Dynkin}).

For a simple Lie algebra $\h$ let $k_\h$ denote $\tr_\h(\alpha^{\vee
2})$ for a long root $\alpha\in\Delta(\h)$. The numbers $k_\h$ for
all simple Lie algebras are given in Table~\ref{Tbl:5.3.7}.

\begin{longtable}{|c|c|c|c|c|c|c|c|c|c|}
\caption{$k_\h$. }\label{Tbl:5.3.7}\\\hline
$\h$&$A_l$&$B_l$&$C_l$&$D_l$&$E_6$&$E_7$&$E_8$&$F_4$&$G_2$\\\hline
$k_\h$& $4l+4$& $8l-4$&$4l+4$&$8l-8$&48&72& 120&36&16\\\hline
\end{longtable}

\begin{Lem}\label{Lem:7.2.2}
Let $\h$ be a nonzero  $\Lambda$-essential subalgebra of $\g$,
$[\h,\h]=\h_1\oplus\ldots\oplus\h_k$ the decomposition into the
direct sum of simple ideals and  $i_j:=i(\h_j,\g),j=\overline{1,k}$.
\begin{enumerate}
\item Suppose  $\g$ is of types $A,B,D,E,F$, $\rank \g>2$,  $W(\g,\h)=W(\g)$
and $\Lambda(\g,\h)=2\Lambda(\g)$ (the last condition is essential
only for $\g\cong \so_{2l+1},F_4$). Then $\h$ is semisimple and
there are positive integers $a_j,j=\overline{1,k}$ such that
\begin{equation}\label{eq:7.2:3}
\begin{split}
&\sum_{j=1}^k a_ji_j=4,\\
&\sum_{j=1}^k a_jk_{\h_j}=2k_\g-16.
\end{split}
\end{equation}
\item Suppose $\g$ is of type $C_l,l>2,F_4$,  $\h$ is a $\Lambda$-
essential subalgebra of $\g$. Then $\Lambda(\g,[\h,\h])\neq
\Lambda(\g)$. In other words, $\h$ contains a nonzero
$\Lambda$-essential semisimple ideal. Suppose, in addition, that
$\h$ is semisimple. Then there are nonnegative integers
$a_j,j=\overline{1,k}$ such that
\begin{equation}\label{eq:7.2:4}
\begin{split}
&\sum_{j=1}^k a_ji_j=8,\\
&\sum_{j=1}^k a_j k_{\h_j}=4k_{\g}-16.
\end{split}
\end{equation}
Further, if any proper ideal of $\h$ is not $\Lambda$-essential,
then $a_j>0$ for any $j$ and there is a subalgebra $\s\subset\h,
\s\sim_G\so_3\subset \sl_3\cong \g^{(\alpha_1,\alpha_2)},
\alpha_1,\alpha_2\in \Delta(\g)^{min}$, such that $\s$ is not
contained in a proper ideal of $\h$.
\item Suppose $\g$ is of types $A,C-F$. Then there are  $h\in\h$ satisfying
(\ref{eq:7.2:1}) and a subalgebra $\s\subset\h,\s\cong\sl_2,$ such
that $h\sim_G\alpha^\vee,\alpha\in \Delta(\g)^{min},h\in\s$ and
$\s$ is not contained in a proper ideal of $[\h,\h]$.
\end{enumerate}
\end{Lem}
\begin{proof}
From \cite{Weyl}, Theorem 5.1.2, it follows that $s_\alpha\in
W(\g,\h)$ for $\alpha\in \Delta(\g)^{min}$ provided
$\g=F_4,\sp_{2l}$. When $\g\neq \so_{2l+1}$ there are
$\alpha_1,\alpha_2\in \Delta(\g)^{min}$ such that
$\g^{(\alpha_1,\alpha_2)}\cong\sl_3$. When $\g=\so_{2l+1},l>2, F_4$
there are $\alpha_1,\alpha_2\in \Delta(\g)$ with
$\g^{(\alpha_1,\alpha_2)}\cong\sl_3$. Applying
Proposition~\ref{Prop:7.1.8}, we see that
$\widetilde{R}^{(\alpha_1,\alpha_2)}\rightsquigarrow_\g T^*(G/H)$.
Let $\s$ denote a subalgebra in $\h$ such that $\s\sim_G\so_3\subset
\g^{(\alpha_1,\alpha_2)}$ and $U$ the 5-dimensional irreducible
$\s$-module. Then $(\s,U)\rightsquigarrow_\h \g/\h$. Denote by
$\h^1$ the ideal in $\h$ generated by $\s$. Clearly,
$(\s,U)\rightsquigarrow_{\h^1}\g/\h^1$. This implies assertion 3.

Since $(\s,U)\rightsquigarrow_{\h^1}\g/\h^1$, we have
$\widetilde{R}^{(\alpha_1,\alpha_2)}\rightsquigarrow_\g T^*(G/H^1)$.
Set  $a_j:=\iota(\s_j,\h_j)$, where $\s_j$ denotes the projection of
$\s$ to $\h_j$. Thanks to Proposition~\ref{Prop:7.1.8},
$\Lambda(\g,\h^1)$ equals $2\Lambda(\g)$ in assertion 1, and is
contained $\Span_\Z(2\Delta(\g)^{min}\cup \Delta(\g)^{max})$ in
assertion 2. Hence if $\Lambda(\g,\h)=\Span_\Z(2\Delta(\g)^{min}\cup
\Delta(\g)^{max})$, we get $\h^1=\h$, whence $a_j>0$ for all $\g$.
In assertion 1 the equality $i(\s,\g)=4$ holds, and in assertion 2
we have $i(\s,\g)=8$. The first equalities in  (\ref{eq:7.2:3}),
(\ref{eq:7.2:4}) follow from the additivity property \cite{Weyl},
(5.1), of the Dynkin index proved in \cite{Dynkin}.  The
$\s$-modules $\g/\h$, $\h/\s\oplus U$ differ by a trivial summand.
One gets the second equalities in (\ref{eq:7.2:3}), (\ref{eq:7.2:4})
by computing the traces of $h^2$, where $h$ is a coroot in
$\s\cong\so_3$, on these modules.
\end{proof}

Now we prove some statements concerning reductive subalgebras in
classical Lie algebras containing an element conjugate to
$\alpha^\vee,\alpha\in \Delta(\g)$.

\begin{Prop}\label{Prop:7.2.5}
Let $\g$ be a classical Lie algebra and  $\h$ a reductive subalgebra
$\g$ such that there is $h\in\h$ such that $h\sim_G \alpha^\vee$ for
$\alpha\in\Delta(\g)$ and $h$ is not contained in a proper ideal of
$\h$.
\begin{enumerate}
\item If $\g=\sl_n$ and $\h$ is semisimple, then $\h=\sl_k,\so_k,\sp_k$.
\item  If $\g=\so_{2n+1}$ and $\alpha\in \Delta(\g)^{max}$, then
$\h=\so_k,\gl_k^{diag},\spin_8$.
\item Suppose $\g\cong\so_{2n}$,  $\h$ is semisimple, and
$h$ is included into an  $\sl_2$-triple in $\h$.  Then $\h=\so_k,
\so_k\oplus\so_l, \sl_k^{diag},\sp_k^{diag},\so_k^{diag},
\spin_7,G_2,\spin_8$.
\item Suppose $\g=\sp_{2n}$, $\h$ is semisimple, and $\alpha\in
\Delta(\g)^{min}$. Then $\h=\sp_{2k}, \sp_{2k}\oplus \sp_{2l},
\sl_k^{diag},\so_k^{diag},\sp_k^{diag}$.
\end{enumerate}
\end{Prop}
\begin{proof}
Let  $V$ denote the tautological  $\g$-module.

{\it Step 1.} Here we describe all semisimple subalgebras
$\h\subset\gl(V)$ containing  $h\in\gl(V)$ such that:
\begin{itemize}
\item[(a)]  $h$ is semisimple and its eigenvalues are $\pm 1$, each of multiplicity 1, and 0
of multiplicity $\dim V-2$.
\item[(b)]  $h$ is not contained in a proper ideal of $\h$.
\end{itemize}

 Since $\tr_U\xi=0$ for any $\h$-module
$U$ and $\xi\in\h$, we see that $V/V^\h$ is an irreducible
$\h$-module. All irreducible linear algebras $\h$ containing such
$h$ where described in Proposition 8 from~\cite{wc}. These are
$\sl(V/V^\h),\so(V/V^\h)$ and $\sp(V/V^\h)$.

{\it Step 2.} Here we describe all reductive subalgebras $\h\subset
\so(V)$ containing $h$ satisfying (a),(b). If $\h$ is semisimple,
then $\h=\so_k$ by step 1. Suppose $\h$ is not semisimple. In this
case $V/V^\h$ is reducible.  Analogously to step 1, there is no
proper orthogonal submodule in $V/V^\h$. Therefore there is an
irreducible $\h$-module $V_0$ such that $V/V^\h=V_0\oplus V_0^*$. We
may assume that $h$ acts on $V_0$ as $diag(1,0,\ldots,0)$. By
\cite{Vin_Heis}, Proposition 2, $\h=\gl_k$.

{\it Step 3.} Here we classify all irreducible subalgebras
$\h\subset\gl(V)$ such that the $\h$-module $V$ is self-dual and
there is $h\in \h$ that satisfies (b) and
\begin{itemize}
\item[(a$'$)]  $h$ is semisimple and its eigenvalues are $\pm 1$ of
multiplicity  2 each and $0$ of multiplicity $\dim V-4$.
\end{itemize}

Choose a Cartan subalgebra $\t\subset\h$ containing $h$. We may
assume that the positive root system $\Delta(\h)_+$ is chosen in
such a way that
  $\langle \Delta(\h)_+,h\rangle\geqslant 0$.
Let $\lambda_1,\ldots,\lambda_k$ be all different dominant weights
of the $\h$-module $V$, where $\lambda_1$ is the highest weight. Let
us check that $k=1$. Assume the converse. We note that
$\langle\lambda_i,h\rangle>0$ for all $i$. Indeed, being a dominant
coweight, $h$ is the sum of simple coroots with positive
coefficients. Therefore $k=2$ and both $\lambda_1,\lambda_2$ have
multiplicity 1. Besides, as $V$ is self-dual,
 $W(\h)\lambda_i=-W(\h)\lambda_i,i=1,2$. For all $\nu\in W(\h)\lambda_i, \nu\neq \pm\lambda_i,$
 we get $\langle\nu,h\rangle=0$.
 It was shown in~\cite{Vin_Heis}, Lemma 2, that  $\h=\so_n,\sp_n$
 and the weight
 $\lambda_i$ is proportional (up to an automorphism for $\h=\so_8$) to the highest weight of the tautological
$\h$-module. But  $\lambda_1$ and $\lambda_2$ are not proportional,
for $\langle\lambda_i, h\rangle=1$. Thus $\h=\so_8$ and , up to an
automorphism, $\lambda_1=l\pi_1$, where $l>1$. In this case
$(l-2)\pi_1+\pi_2$ is a weight of $V$. Contradiction.

So the highest weight is the only nonzero dominant weight of the
$\h$-module $V$.

At first, let us consider the case $V^{\t}\neq 0$. In this case $\h$
is simple and
 $\lambda$ is the maximal short root.

Suppose $\h$ is of types $A,D,E$. In this case $V=\h$. There are
exactly two positive roots having a nonzero pairing with $h$. These
are the maximal root $\delta$ and another root, say, $\beta$.
Clearly,  $\beta$ is a simple root and any root greater than
 $\beta$ is maximal. Thus $\h=A_2$.

If $\h\cong\so_{2l+1}, l>1$, then $\h$ contains a required element
$h$.

If $\h\cong \sp_{2l},l>2,F_4$, then  $\Delta(\h)^{min}$ is the root
system $D_l$, and, by above, there is no  $h\in\t$ with required
properties.

Finally, let $\h=G_2$. In this case $\Delta(\h)^{min}=A_2$ and $h$
exists.

Now consider the case $V^{\t}= 0$. In this case $\lambda_1$ is
minuscule, that is,  $\langle \lambda_1,\delta^\vee\rangle=1$, where
$\delta^\vee$ denotes the maximal coroot. It follows that
$\lambda_1$ is a fundamental weight, пусть $\pi_m$. There is a
unique weight less than  $\lambda_1$ w.r.t. the natural order on the
set of weights, namely, $\lambda_1-\alpha_m$. This observation makes
possible to find the system of linear equations for $h$. This system
has a solution only in the following cases:

1) $\h=\so_{2n},\sp_{2n}$, $V$ is the tautological $\h$-module (or a
half-spinor module for  $\h=\so_8$).

2) $\h=\so_7, \lambda_1=\pi_3$.

{\it Step 4.} Complete the proof of the proposition. Assertions 1
and 2 were proved on steps 1 and 2, respectively. Assertions 3 and 4
in the case when the  $\h$-module $V/V^\h$ is reducible also follow
from steps  1,2. If $V/V^\h$ is irreducible, the image of $\h$ in
$\gl(V/V^\h)$ is one of the subalgebras found on step 3. All of them
except of $\ad(\sl_3)$ fulfill the condition that    $h$ is included
into an $\sl_2$-triple in $\h$.
\end{proof}

\begin{proof}[Proof of Theorem~\ref{Thm:7.0.2}]
Throughout the proof $\h$ denotes a $\Lambda$-essential subalgebra
of $\g$. Let $H$ denote the connected subgroup of $G$ with Lie
algebra $\h$ and $\widetilde{H}$ denote the inverse image of
$Z(N_G(H)/H)$ in $N_G(H)$.

{\it The case $\g\cong\sl_n$}. At first, suppose $W(\g,\h)\neq
W(\g)$. Let us check that $\Lambda(\g,\h)=\Lambda$, where
$\Lambda:=\Span_\Z(\alpha\in\Delta(\g)| s_\alpha\in W(\g,\h))$. By
Proposition~\ref{Prop:1.4.8}, there is an inclusion
$\Lambda(\g,\h)\subset \Lambda$. To prove the inverse inclusion we
need to check that $\alpha\in \Lambda(\g,\h)$ for all $\alpha\in
\Delta(\g)$ with $s_\alpha\in W(\g,\h)$. Considering case by case
possible groups $W(\g,\h)$ (\cite{Weyl}, Theorem 5.1.2), we note
that there is $\alpha_1\in\Delta(\g)$ such that $\alpha,\alpha_1$
satisfy the assumptions of Corollary~\ref{Cor:7.1.7}. Applying this
corollary, we get
 $\Lambda(\g,\h)=\Lambda$.

Now suppose $W(\g,\h)=W(\g)$. By Lemma~\ref{Lem:7.2.3},
$\Lambda(\g,\h)=2\Lambda(\g)$. Thanks to Lemma~\ref{Lem:7.2.1},
$G\alpha^\vee\cap\h\neq\varnothing$. By Lemma~\ref{Lem:7.2.2}, $\h$
is semisimple. Proposition \ref{Prop:7.2.5} implies
$\h=\sl_k,\so_k,\sp_{2k}$.

If $\h=\sl_k$, then, since $\a(\g,\h)=\t, W(\g,\h)=W(\g)$, we have
$k\leqslant \frac{n}{2}$. By (\ref{eq:7.2:3}), $k=\frac{n}{2}-1$. To
show that $\Lambda(\g,\h)=\Lambda(\g)$ it is enough to note
$\Lambda(\g,\h)\supset \Lambda(\g,\sl_{n/2})$ (see
Proposition~\ref{Prop:1.4.4}).

Let $\h=\sp_{2k}$. Analogously to the previous paragraph, we get
$k=\frac{n}{2}-2$ and
$\Lambda(\g,\h)\supset\Lambda(\g,\sp_{n-2})=\Lambda(\g)$.

Finally, suppose $\h=\so_k$. By (\ref{eq:7.2:3}),  $k=n$. From
\cite{Kramer}, Tabelle 1, it follows that $\X_{\SL_n,\SL_n/\SO_n}=
\Span_\Z(2\pi_1,\ldots,2\pi_{n-1})$. Since
$\X_{\SL_n,\SL_n/\SO_n}\subset \Lambda(\g,\h)$, we get
$\Lambda(\g,\h)=2\Lambda(\g)$.

{\it The case $\g=\so_{2n+1}, n\geqslant 2$}. At first,  we consider
the case $W(\g,\h)\neq W(\g)$. By \cite{Weyl},
$[\h,\h]=\sl_n^{diag}, G_2 (n=4),\spin_7, (n=5)$. Since
$\Lambda(\g)$ is generated by $\Delta(\g)^{min}$, it follows from
Proposition \ref{Prop:7.1.4} that there is $h\in\h, h\sim_G
\alpha^\vee$. By Proposition \ref{Prop:7.2.5}, $\h=\gl_n^{diag}$. By
\cite{Kramer}, Tabelle 1,
$\X_{\SO_{2n+1},\SO_{2n+1}/\GL_n^{diag}}=\Span_\Z\{\varepsilon_i\}|_{i=\overline{1,n}}$.
Assume that $G=\SO_{2n+1}$.  Choose nonzero vectors $v\in
(\C^{2n+1})^H, \omega\in (\bigwedge^2\C^{2n+1})^H$. The spaces
$(\bigwedge^{2i+1}\C^{2n+1})^H,  (\bigwedge^{2i}\C^{2n+1})^H$ are
1-dimensional, for $G/H$ is spherical. These spaces are generated by
$v\wedge\omega^{\wedge i},\omega^{\wedge i}$, respectively.  The
group $N_G(H)/H$ is isomorphic to $\Z_2$. The nontrivial element of
this group acts on $(\bigwedge^i\C^{2n+1})^{H}$ by
$(-1)^{2\{i/2\}n+[i/2]}$. To show that $\Lambda(\g,\h)$ has the
required form we use Proposition~\ref{Prop:3.7.1}.

Now suppose that $W(\g,\h)=W(\g)$. From Lemma~\ref{Lem:7.2.1} it
follows that there is an element $h\in\h$ satisfying
(\ref{eq:7.2:1}) and such that $h\sim_{\SO_{2n+1}}
diag(2,-2,0\ldots,0)$ whence $\tr_\g h^2=16n-8$. By assertion 2 of
Proposition~\ref{Prop:7.2.5},   $h$ is contained in an ideal of $\h$
of the form $\gl_k^{diag}$ or $\so_k$. So we have $\tr_\h
h^2=8(k-1)$ (for $\gl_k^{diag}$) or $\tr_\h h^2=8(k-2)$ (for
$\so_k$). Thus the ideal $\h_1$ of $\h$ generated by $h$ coincides
with $\so_{n+1}$.

 Let us show, at first, that
$\Lambda(\so_{2n+1},\so_{n+1})$ has the form indicated in
Table~\ref{Tbl:7.0.3}. By Lemma \ref{Lem:7.2.2},
$\Lambda(\so_{2n+1},\so_{n+1})\neq 2\Lambda(\g)$, for
(\ref{eq:7.2:3}) is rewritten in the form $4(4n-4)=2(8n-4)-16$.

Suppose $G=\SO_{2n+1}$.  By Proposition~\ref{Prop:3.7.1},
$\Lambda(\g,\h)=\X_{G,G/\widetilde{H}}$. Now it is enough to show
that $L_{0\,G,G/\widetilde{H}}\neq \{1\}$. As Knop proved
in~\cite{Knop1}, Korollar 8.2, $L_{0\,G,G/\widetilde{H}}$ is the
stabilizer in general position for the action
$G:T^*(G/\widetilde{H})$. So it remains to show that the stabilizer
in general position for the action $\widetilde{H}:\g/\h$ is
nontrivial. This action coincides with the action of $\O(n+1)$ on
$(\C^{n+1})^{\oplus n}$. The stabilizer in general position is
isomorphic to  $\Z_2$.

Now suppose $\h\neq\h_1=\so_{n+1}$. By above,
$\Lambda(\g,\h)=2\Lambda(\g)$. Assertion 1 of Lemma~\ref{Lem:7.2.2}
implies that $\h$ is semisimple. By Lemma~\ref{Lem:7.2.1}, there is
$h\in\h$ such that and
\begin{equation}\label{eq:7.2:10}k_\g=\tr_\g h^2=2\tr_\h h^2+8\end{equation}
and $h\sim_{\SO_{2n+1}}diag(1,1,-1,-1,0\ldots,0)$.  As we checked in
the proof of the inequality $\Lambda(\g,\h_1)\neq 2\Lambda(\g)$,
$h\not\in\h_1$. Further,  $h\not\in\h_1^\perp$. Otherwise, $\tr_\h
h^2=\tr_{\z_\g(\h)}h^2$ and (\ref{eq:7.2:10}) does not hold.

Let us check that if $\h\neq
\widehat{\h}:=\n_\g(\h_1)=\so_n\oplus\so_{n+1}$, then  $\tr_\h
h^2<\tr_{\widetilde{\h}} h^2$. Otherwise, $h$ acts trivially on
$\widehat{\h}/\h$, for $2\tr_{\widehat{\h_1}} h^2=k_\g-8$. So $\h$
and $\widehat{\h}$ have a common ideal containing $h$. Since
$h\not\in\h_1$, this is impossible
 whence $\h=\so_{n+1}\oplus
\so_n$. Let us check that $\Lambda(\g,\h)=2\Lambda(\g)$. Indeed,
$N_G(H)=\SL_{2n+1}\cap (\O_{n}\times \O_{n+1})$,
$\Lambda(\g,\h)=\X_{G,G/N_G(H)}$, and the last lattice is easily
extracted from \cite{Kramer}, Tabelle 1.

{\it The case $\g=\sp_{2n}, n>2$.}

At first, we determine all  nonzero $\Lambda$-essential subalgebras
$\h\subset\g$ whose proper ideals are not $\Lambda$-essential.
Thanks to assertion 2 of Lemma \ref{Lem:7.2.2}, we see that $\h$ is
semisimple. Applying Theorem 5.1.2 from \cite{Weyl}, we see that
$W(\g,\h)$ contains $s_\alpha$ for any $\alpha\in \Delta(\g)^{min}$.
By assertion 2 of Lemma~\ref{Lem:7.2.1}, there is a subalgebra
$\s\subset\h$ not contained in a proper ideal of $\h$ such that
$\s\sim_{\Sp_{2n}} \so_3^{diag}$. Taking into account  assertion 3 of 
Proposition~\ref{Prop:7.2.5}, we see that
$\h=\sp_{2k},\sl_k^{diag},\so_k^{diag}$. Only
$\sl_n^{diag},\sp_{n-2}$ satisfy (\ref{eq:7.2:4}).

Let us show that $\Lambda(\sp_{4m+2},\sp_{2m})=\Lambda(\g)$. Set
$G=\Ad(\Sp_{4m+2})$. The center $N_G(H)/H$ is $\{1\}$. By
Proposition~\ref{Prop:3.7.1}, $\Lambda(\g,\h)=\X_{G,G/H}$. So it
remains to prove that the s.g.p. for the action $H:\g/\h$ is
trivial. The last action coincides with the natural action of
$\Sp_{2m}$ on  $(\C^{2m})^{\oplus 2m+2}$. But the s.g.p. is trivial
already for $\Sp_{2m}:(\C^{2m})^{\oplus 2m}$.

Suppose $\h=\sl_n^{diag}$. Let us check that
$\Lambda(\g,\h)=\Span_\Z(2\Delta(\g)^{min}\cup \Delta(\g)^{max})$.
At first, we show that
$\Lambda(\sp_{2n},\gl_n^{diag})=2\Lambda(\g)$. Indeed the subalgebra
$\gl_n^{diag}\subset\sp_{2n}$ is spherical, and the subgroup
$\GL_n^{diag}\subset \Sp_{2n}$ is of index  2 in its normalizer. By
Proposition~\ref{Prop:3.7.1}, $\Lambda(\sp_{2n},\gl_n^{diag})$ is
also of index 2 in $\X_{\Sp_{2n},\Sp_{2n}/\GL_n^{diag}}$. By
\cite{Kramer}, the last lattice equals
$\Span_\Z(2\pi_1,\ldots,2\pi_n)$ whence the equality for
$\Lambda(\sp_{2n},\gl_n^{diag})$.

By assertion 2 of Lemma \ref{Lem:7.2.2}, $\Lambda(\g,\h)\subset
\Span_\Z(2\Delta(\g)^{min}\cup\Delta(\g)^{max})$. The equality will
follow if we check that $\Lambda(\g,\h)\neq 2\Lambda(\g)$. Assume
the converse. By Lemma \ref{Lem:7.2.1}, there is $h_0\in \h,
h_0\sim_{\Sp_{2n}}diag(1,-1,0,\ldots,0)$, which is absurd.

It remains to prove that  $\gl_n\subset \sp_{2n}$ is the only
subalgebra $\h$ satisfying $\Lambda(\g,\h)=2\Lambda(\g)$. Indeed,
let  $\h$ be such a subalgebra. By assertion 2 of
Lemma~\ref{Lem:7.2.2}, $\Lambda(\g,[\h,\h])\neq \Lambda(\g)$. Thus
$[\h,\h]=\sl_n$.

{\it The case $\g=\so_{2n},n\geqslant 8$.} Let $\h$ be a
$\Lambda$-essential subalgebra of $\g$. By assertion of 1
Lemma~\ref{Lem:7.2.2},  $\h$ is semisimple. According to
Lemma~\ref{Lem:7.2.1}, $\h$ contains an element $h$ satisfying
(\ref{eq:7.2:1}) such that
$h\sim_{\SO_{2n}}diag(1,1,-1,-1,0\ldots,0)$. By assertion 3 of
Lemma~\ref{Lem:7.2.2}, we may assume that some multiple of  $h$ can
be included into an  $\sl_2$-triple not contained in a proper ideal
of $\h$. Assertion 4 of Proposition \ref{Prop:7.2.5} implies that
$\h=\so_k,\so_k\oplus\so_l,
\sl_k^{diag},\sp_k^{diag},\so_k^{diag},\spin_7,$ $G_2$. We note that
$\h\neq\sl_n,\so_k,k>n$, for $\a(\g,\h)=\t$.

Firstly, consider the case $\h=\so_k\oplus \so_l$, where
$k,l\leqslant n$. Here the projection of  $h$ to both ideals
$\so_k,\so_l$ is conjugate to $diag(1,-1,0,\ldots,0)$.
(\ref{eq:7.2:1}) holds iff $k=l=n$. Let us check that indeed
$\Lambda(\g,\h)=2\Lambda(\g)$. We may assume  $G=\SO_{2n}$. The
homogeneous space $G/H$ is spherical and $\#N_G(H)/H=2$. Thus
$\Lambda(\g,\h)=\X_{G,G/N_G(H)}$ is of index 2 in $\X_{G,G/H}$. The
required equality follows easily from Tabelle 1 of~\cite{Kramer}.

Among the remaining subalgebras $\h$  only
$\sl_{n-2}\subset\so_{2n}$, $\spin_7\subset\so_{12},G_2\subset
\so_{10}$ satisfy (\ref{eq:7.2:3}). Let us check that in these cases
$\Lambda(\g,\h)=\Lambda(\g)$. For $\sl_{n-2}$ this stems from the
inclusion $\sl_{n-2}\subset\sl_{n-1}$. In the other cases take
$\Ad(\SO_{2m})$ for $G$. Note that $N_G(H)/H\cong \Ad(\SO_k)$ for
$k=3,4$. By Proposition \ref{Prop:3.7.1},
$\Lambda(\g,\h)=\X_{G,G/H}$. To prove the equality
$\X_{G,G/H}=\Lambda(\g)$ it is enough to check that the s.g.p. for
the action $H:\g/\h$ is trivial. This follows from the
classification of Popov, \cite{Popov_stab}.

{\it The case $\g=E_6$.} By Lemma~\ref{Lem:7.2.2}, $\h$ is
semisimple. Let $\h_j,i_j,a_j, j=\overline{1,k},$ be such as in
assertion 1 of Lemma~\ref{Lem:7.2.2}. We reorder $\h_j$ in such a
way that $k_{\h_1}\geqslant k_{\h_2}\geqslant \ldots\geqslant
k_{\h_k}$. Since $\a(\g,\h)=\t$, we have $\h_i\neq D_5,A_5,B_4,F_4$.
(\ref{eq:7.2:3}) can be rewritten as
\begin{equation}\label{eq:7.2:5}
\begin{split}&\sum_{j=1}^k a_ji_j=4,\\
&\sum_{j=1}^k a_j k_{\h_j}=80.
\end{split}
\end{equation}
Thus $k_{\h_1}\geqslant 20$. It follows that $\h_1$ is one of the
subalgebras $A_4,B_3,C_4,D_4\subset E_6$.

If $\h_1=D_4$, then $\h=\h_1$, for $\n_\g(\h_1)/\h_1$ is
commutative. This contradicts (\ref{eq:7.2:5}).  The subalgebra
$\h_1=B_3$ is embedded into $D_4$ and
$\h_1=[\n_\g(\h_1),\n_\g(\h_1)]$ so in this case
$\Lambda(\g,\h)=\Lambda(\g)$ too.

Consider the case $\h_1=A_4$. It is easy to see that
$\n_\g(\h_1)/\h_1\cong \C\times\sl_2$. If $\h\neq \h_1$, then
$\h=A_4\times A_1$. However in this case (\ref{eq:7.2:5}) has no
positive solutions. So $\h=A_4$. Take $\Ad(E_6)$ for $G$. The
subalgebra $\h$ is included into $D_5$. So $N_G(\h)$ acts on $\h$ as
$\Aut(\h)$. Clearly, $N_G(H)^\circ$ is a Levi subgroup of $G$. From
this we deduce that $N_G(H)$ has exactly two connected components.
Choose $\sigma\in N_G(H)\setminus N_G(H)^\circ$. Let $Z,F$ denote
the center and the commutant of  $(N_G(H)/H)^\circ$, respectively.
The element $\sigma$ acts on $\z$ by $-1$. Therefore the image of
$Z(N_G(H)/H)$ under the projection $(N_G(H)/H)^\circ\rightarrow
(N_G(H)/H)^\circ/F$ is isomorphic to $\Z_2$. On the other hand, the
center of $F$ is of order  at most 2. So $\# Z(N_G(H)/H)\leqslant
4$. By Proposition~\ref{Prop:3.7.1},
$\X_{G,G/\widetilde{H}}=\Lambda(\g,\h)$. If $\Lambda(\g,\h)\neq
\Lambda(\g)$, then $L_{0\,G,G/\widetilde{H}}\cong
\Lambda(\g)/2\Lambda(\g)\cong \Z_2^6$. So the s.g.p. for the action
$\widetilde{H}:\g/\h$ is isomorphic to $\Z_2^6$. Therefore the
s.g.p. for the action $H:\g/\h$ is nontrivial. Clearly, $\g/\h\cong
(\bigwedge^2 \C^5\oplus \bigwedge^2\C^{5*}\oplus\C)^{\oplus 2}\oplus
\C^5\oplus\C^{5*}$. By~\cite{Popov_stab}, the s.g.p. for this action
is trivial.

Finally, let us consider the case $\h=C_4$. In this case the
inequality $\X_{G,G/H}\neq\Lambda(\g)$ stems from \cite{Kramer},
Tabelle 1.

{\it The case $\g=E_7$.} In this case, by Lemma~\ref{Lem:7.2.2},
$\h$ is semisimple. Define $\h_j,i_j,a_j, j=\overline{1,k},$
analogously to the previous case. Since $\a(\g,\h)=\t$, $\h_i\neq
E_6,D_6$. (\ref{eq:7.2:3}) is rewritten as
\begin{equation}\label{eq:7.2:6}
\begin{split}&\sum_{j=1}^k a_ji_j=4,\\
&\sum_{j=1}^k a_j k_{\h_j}=128.
\end{split}
\end{equation}
Thus $k_{\h_1}\geqslant 32$. It follows that $\h_1=A_7,D_5,B_5,F_4$.
If $\h_1=B_5,F_4$, then $k_{\h_1}=36$. Therefore $\h\neq\h_1,
a_1\leqslant 3$, and $k_{\h_2}\geqslant 20$. One easily sees that
this is impossible. If $\h_1=D_5$, then $k_{\h_1}=32$ whence
$\h=\h_1$. But $D_5$ is included into $B_5$ whence
$\Lambda(E_7,D_5)=\Lambda(\g)$. Finally, for $\h=A_7$ the inequality
$\Lambda(\g,\h)\neq \Lambda(\g)$ follows from \cite{Kramer}, Tabelle
1.

{\it The case $\g=E_8$.} Again, $\h$ is semisimple. Let
$\h_j,i_j,a_j, j=\overline{1,k}$ be such as in the case $E_6$. Note
that $\h_1\neq E_7$.  (\ref{eq:7.2:3}) is rewritten as
\begin{equation}\label{eq:7.2:7}
\begin{split}&\sum_{j=1}^k a_ji_j=4,\\
&\sum_{j=1}^k a_j k_{\h_j}=224.
\end{split}
\end{equation}

Therefore $k_{\h_1}\geqslant 56$. Hence $\h_1=D_8$. The inequality
$\Lambda(\g,\h)\neq \Lambda(\g)$ follows from \cite{Kramer}, Tabelle
1.

{\it The case $\g=F_4$.} At first, suppose that $\h$ does not have
nonzero  $\Lambda$-essential ideals. By assertion 2 of Lemma
\ref{Lem:7.2.2}, $\h$ is semisimple. Let $a_j,i_j,k_{\h_j}$ have the
same meaning as in the case $\g=E_6$. Since $\a(\g,\h)=\t$, we see
that $\h\neq B_4,D_4$. (\ref{eq:7.2:4}) can be rewritten as
\begin{equation}\label{eq:7.2:8}
\begin{split}
&\sum_{j=1}^k a_ji_j=8,\\
&\sum_{j=1}^k a_j k_{\h_j}=128.
\end{split}
\end{equation}
It follows that $k_{\h_1}\geqslant 16$. Thus $\h_1=A_3,C_3,B_3$ or $
G_2$. If $\h_1=A_3$, then  $k_{\h_1}=16$. Thus $a_1=8$ and
$\h=\h_1$. However  $A_3$ does not contain a subalgebra
$\s\cong\sl_2$ of index 8, contradiction with assertion 2 of
Lemma~\ref{Lem:7.2.2}.

Suppose $\h_1=B_3$. Since $k_{\h_1}=20$, we see that $\h\neq \h_1$.
This contradicts  $\h_1=[\n_\g(\h_1),\n_\g(\h_1)]$. So
$\Lambda(F_4,B_3)=\Lambda(\g)$. If $\h_1=G_2$, then again
$\h_1=[\n_\g(\h_1),\n_\g(\h_1)]$ whence $\h=\h_1$. However, $G_2$ is
included into $B_3$, contradiction.

Finally, consider the case  $\h_1=C_3$. Since $k_{\h_1}=16$, we have
$\h=\h_1$. The  $H$-modules $\g/\h\cong V(\pi_3)^{\oplus 2}$ are
isomorphic. As Popov proved in~\cite{Popov_stab}, the s.g.p. for the
action $\Sp(6):V(\pi_3)^{\oplus 2}$ is isomorphic to $\Z_2\times
\Z_2$. Thus $L_{0\,G,G/H}\cong \Z_2\times \Z_2$ whence
$\Lambda(\g,\h)\neq \Lambda(\g)$. By Theorem 5.1.2 from \cite{Weyl},
$W(\g,\h)=W(\g)$. Since the system (\ref{eq:7.2:3}) has no solution,
we get $\Lambda(\g,\h)\neq 2\Lambda(\g)$.

Now let us determine all $\Lambda$-essential subalgebras
$\h\subset\g$ that have the ideal $C_3$. In particular,
$W(\g,\h)=W(\g)$. By assertion 1 of Lemma ~\ref{Lem:7.2.2},  $\h$ is
semisimple. The only possibility is $\h=C_3\times A_1$. The equality
$\Lambda(\g,\h)=2\Lambda(\g)$ follows from \cite{Kramer}, Tabelle 1.

{\it The case $\g=G_2$.} By Lemma \ref{Lem:7.2.1},
$\Lambda(\g,\h)=2\Lambda(\g)$ and there is $h\in\h,
h\sim_G\alpha^\vee,\alpha\in \Delta(\g)^{min}$ such that
\begin{equation}\label{eq:7.2:9}
\tr_\h h^2=4.\end{equation} Since $\a(\g,\h)=\t$, we have $\h\neq
A_2$. Note that $\h\cap G\beta^\vee\neq\varnothing$ for $\beta\in
\Delta(\g)^{max}$. Thus $\rank\h=2$ and we may assume that
$\t\subset \h$. There are three (up to sign) elements in $\t$ that
are $G$-conjugate with $\alpha^\vee$. Considering them case by case,
we see that if (\ref{eq:7.2:9}) holds, then $\h=A_1\times
\widetilde{A}_1$. Here $\Lambda(\g,\h)=2\Lambda(\g)$ stems
from~\cite{Kramer}, Tabelle 1.
\end{proof}

\subsection{Proof of Theorem~\ref{Thm:7.0.5}}\label{SUBSECTION_root4}
At first, we check that $H^{\X-sat}\subset N_G(H)$. By
Lemma~\ref{Lem:1.3.4}, $H^{\X-sat}/\widehat{H}\subset
Z(N_G(\widehat{H})/\widehat{H})$, where the subgroup
$\widehat{H}\subset N_G(H)$ was defined in Subsection
\ref{SUBSECTION_root_intro} before Theorem~\ref{Thm:7.0.5}.  It is
obvious that $H\subset N_G(\widehat{H})$ whence the claim. Note also
 that $H_0=H\cap H^{\X-sat}$ is a normal subgroup in $H$.

Now let us show that $\X_{G,G/H}=\X_{G,G/H_0}$. By the definition of
$H_0$, we have $\Lambda(\g,\h_0)=\Lambda(\g,\h)\subset
\X_{G,G/H}\subset \X_{G,G/H_0}$. The inclusion
$\X_{G,G/H}/\Lambda(\g,\h)\hookrightarrow
\X_{G,G/H_0}/\Lambda(\g,\h)$ corresponds to the epimorphism
$\A_{G,G/H_0}\twoheadrightarrow \A_{G,G/H}$ (existing by
Proposition~\ref{Prop:1.4.5}). It remains to check that the kernel
of the last homomorphism is trivial. By the uniqueness part of
Proposition~\ref{Prop:1.4.5}, the homomorphism
$\A_{G,G/H_0}\rightarrow \Aut^G(G/H)\cong N_G(H)/H$ coincides with
$H^{\X-sat}/H_0\rightarrow N_G(H)/H$. The latter is injective, for
$H_0=H\cap H^{\X-sat}$.

Let us show that $\h_0=[\h^{\X-sat},\h^{\X-sat}]$. This will
immediately yield $H_0^{\X-sat}=H^{\X-sat}$. Let us note that
$[\h^{\X-sat},\h^{\X-sat}]=\widehat{\h}\subset\h$ by the definition
of $H^{\X-sat}$. The required equality follows from the observation
that $\widehat{\h}$ is the maximal ideal of $\h^{\X-sat}$ such that
$\a(\g,\widehat{\h})=\t$.

Proceed to the proof of assertion 2. The only nontrivial claim here
is the particular form of the duality for algebras $\h$ from
Table~\ref{Tbl:7.0.3}. The Frobenius reciprocity implies that
$\X_{G,G/H^\circ}$ is spanned by all $\lambda$ with
$V(\lambda^*)^\h\neq \{0\}$ (as $\h$ is reductive, the latter is
equivalent to $V(\lambda)^\h\neq \{0\}$). Besides, if
$V(\lambda)^\h\neq \{0\}$, then $\chi_\lambda$ coincides with the
character by that $H^{\X-sat}/H^\circ$ acts on $V(\lambda^*)^\h$.
Hence the claim on the restriction of $\chi_\lambda$ to
$Z(G)/(Z(G)\cap H^\circ)$.

All subalgebras except  NN2,4,7,11 are spherical and the lattices
$\X_{G,G/H^\circ}$ were computed in \cite{Kramer}, Tabelle 1. In
cases 2,4 the lattice $\X_{G,G/H^\circ}$ coincides with  $\X(G)$.
This easily follows from observations of the previous paragraph. In
case 7 the lattice is extracted from tables in~\cite{Panyushev4}.
Finally, in case 11 we have seen in the proof of
Theorem~\ref{Thm:7.0.2} that $L_{0\,G,G/H^\circ}\cong \Z_2\times
\Z_2$. Since $\X_G/\Lambda(\g,\h)\cong \Z_2\times \Z_2$, we have
$\X_{G,G/H^\circ}=\Lambda(\g,\h)$. In all cases in consideration
$H^{\X-sat}/H^\circ=Z(N_G(H^\circ))/H^\circ$. Finding a generator (a
typical element in cases 2,3) is not difficult. The character
$\chi_\lambda$ is computed by using the remarks of the previous
paragraph.

Assertion  3 follows directly from the definition of the duality.

\subsection{Finding distinguished components}\label{SUBSECTION_root5}
Below in this subsection  $X=G/H$ and $\underline{X}$ is the
distinguished component of $X^{L_{0\,G,X}}$. To find a point from
$\underline{X}$ we use the following proposition (compare with
\cite{Weyl}, Proposition 4.1.3).

\begin{Prop}\label{Prop:7.5.1}
We use the notation established in Subsection
\ref{SUBSECTION_root_intro}.
\begin{enumerate}
\item Let $H_0=H\cap H^{\X-sat}$, $\widetilde{X}=G/H_0$, $\pi:\widetilde{X}_0\rightarrow
X$ be the natural morphism and $\underline{\widetilde{X}}$ the
distinguished component of $\widetilde{X}^{L_{0\,G,X}}$. Then
$\pi(\underline{\widetilde{X}})\subset X$.
\item Suppose $H\subset H^{\X-sat}$, and let $\pi:G/H\twoheadrightarrow
G/H^{\X-sat}$ be the natural epimorphism and $\underline{X}'$  the
distinguished component of $(G/H^{\X-sat})^{L_{0\,G,G/H^{\X-sat}}}$.
Then $\pi^{-1}(\underline{X}')\subset \underline{X}$.
\item Let $G=G_1\times G_2, H=H_1\times H_2$. Then $\underline{X}$
coincides with the product of the distinguished components of
$(G_i/H_i)^{L_{0\,G_i,G_i/H_i}}$.
\item Suppose $\g$ is simple and $\h$ is the subalgebra from
Table \ref{Tbl:7.0.3} embedded into $\g$ as indicated below. Then
$eH^{\X-sat}$ lies in the distinguished component of
$(G/H^{\X-sat})^{L_{0\,G,G/H^{\X-sat}}}$.
\end{enumerate}
\end{Prop}

Let us describe the embeddings $\h\rightarrow\g$ in consideration.
In cases 1,5,8-10,12-15 we embed $\h$ into $\g$ as the fixed-point
subalgebra of a Weyl involution $\sigma$, i.e., an involutory
automorphism of $\g$ fixing  $\t$ and acting on $\t$ by $-1$. An
involution $\sigma$ is defined  uniquely up to $T$-conjugacy.

In cases 2,3 the embedding $\h\hookrightarrow\g$ is such as in
\cite{Weyl}, Subsection 4.4 .

In case 4 we embed $\h$ into $\g$ as the annihilator of the vectors
$e_{i}+e_{2n+2-i}, i=\overline{1,n}$.

In case 6 we embed $\h$ into  $\g$ as the stabilizer of the
isotropic subspaces $U_{\pm}$ spanned by vectors of the form $x\pm i
\iota(x)$, where $x\in \Span_\C(e_{2i},i\leqslant i)$ and $\iota$ is
an isometrical embedding $\Span_\C(e_{2i})$ into $\Span_\C(e_{2i-1},
i\leqslant n+1)$.

In cases 7,11 $\h=\g^{(\alpha_1,\ldots,\alpha_n-1)},
\g^{(\alpha_1,\alpha_2,\alpha_3)}$, respectively.

\begin{proof}[Proof of Proposition \ref{Prop:7.5.1}]
To prove assertions 1-3 one argues exactly as in the proof of the
analogous assertions of \cite{Weyl},Proposition 4.1.3, (the group
$L_{0\,\bullet,\bullet}^\circ$ there should by replaced with
$L_{0\,\bullet,\bullet}$). Let us prove assertion 4.

Suppose that $\h=\g^\sigma$. Then $\h\oplus\b=\g$,
$\b\cap\h^\perp=\t$. It follows that $B\cap N_G(\h)\subset T$ whence
$A_{G,G/N_G(\h)}\cong T/T\cap H$. The last equality is equivalent to
$T\cap H=L_{0\,G,G/H^{\X-sat}}$ (recall that in the cases in
interest $H$ is spherical whence $H^{\X-sat}=N_G(\h)$). Since the
$B$-orbit of $eH^{\X-sat}$ is dense in $G/H^{\X-sat}$, we see that
$eH$ is contained in the distinguished component.

In cases $2,3$ we get $L_{0\,G,G/H^{\X-sat}}\cong\C^\times$ and the
claim follows from \cite{Weyl}, Proposition 4.1.3.

Below we suppose $H=H^{\X-sat}$ and set
$L_0:=L_{0\,G,G/H^{\X-sat}}$. According to \cite{Weyl}, Proposition
4.3.2, we only need to check that $L_0\subset H$,  $\dim G-\dim
N_G(L_0)=2(\dim H-\dim N_H(L_0))$, and  $N_G(L_0)=N_G(L_0)^\circ
N_{H}(L_0)$.

In case 4 $L_{0}=\{diag(-1,\ldots,-1,1,-1,\ldots,-1)\}$ (here and in
the next case we assume $G=\SO(2n+1)$). Therefore $L_0\subset H$.
Further, $N_G(L_0)\cong\operatorname{S}(\O_{2n}\times\O_1)$,
$N_H(L_0)\cong \operatorname{S}(\O_n\times \O_1)$ whence the two
remaining equalities.

In case 6 $L_0=\{diag(\pm 1,\mp 1,\ldots,-1,1,-1,\ldots,\mp 1,\pm
1)\}$. The nontrivial element of $L_0$ transposes $\GL(n)$-stable
isotropic subspaces whence $L_0\subset H$. Further, $N_G(L_0)\cong
\operatorname{S}(\O_{n+1}\times \O_n)$. On the other hand, embed
$\O_n\cong \O(\Span_\C(e_{2i}))$ into $\O_{2n+1}$ so that $\O_n$
acts on $\Span_\C(e_{2i})$ in the initial way, on $\im\iota$ via
$\iota$ and on the orthogonal complement of the sum of these  two
spaces trivially. Then $\O(n)=N_{H^\circ}(L_0)$ and the required
equalities follow.

In case 7
$L_0=diag(\varepsilon_1,\ldots,\varepsilon_n,\varepsilon_n,\ldots,\varepsilon_1)$,
where $\varepsilon_i\in\{\pm 1\}$ (we consider $G=\Sp_{2n}$).
Clearly, $L_0\subset H$. Further, $N_G(L_0)^\circ\cong \SL_2^n$,
while $N_H(L_0)^\circ$ is a maximal torus of $H$. The groups
$N_G(L_0)/N_G(L_0)^\circ, N_H(L_0)/N_H(L_0)^\circ$ acts on $L_0$ as
the symmetric group on $n$ elements.

In case 11 the group  $L_0$ is generated by $\exp(\pi i
\alpha_1^\vee), \exp(\pi i \alpha_2^\vee)$. The equalities
$\g^{L_0}=D_4, \h^{L_0}\cong\sl_2^3$ hold. Any component of
$N_G(L_0)$ contains an element of $N_G(T)$. It remains to note that
the Weyl group $W(D_4)$ is normal in   $W(\g)$ and
$W(\h)W(D_4)=W(\g)$.
\end{proof}

\section{Algorithm}\label{SECTION_algorithm}
Here we provide an algorithm computing the weight lattice
$\X_{G,X}$, where $X$ is a homogeneous space or an affine
homogeneous vector bundle.

{\it Case 1.} $X=G/H$, where $H$ is a reductive subgroup of $G$ and
$\rank_G(G/H)=\rank(G)$. The lattice $\X_{G,G/H}$ is computed as
indicated in Theorem \ref{Thm:7.0.5}. To compute the distinguished
component of $(G/H)^{L_{0\,G,G/H}}$ we use
Proposition~\ref{Prop:7.5.1}.

{\it Case 2.} Suppose $X=G/H$, where $H$ is a reductive subgroup of
$G$. Using Theorem 1.3, \cite{ranks}, we compute  $\a(\g,\h)$. Then,
applying Proposition 4.1.3 from \cite{Weyl}, we compute a point in
the distinguished component of
$(G/H)^{L_0},L_0:=L_{0\,G,G/H}^\circ$. We may assume that $eH$ lies
in that distinguished component.
 Then applying \cite{Weyl}, Proposition 4.1.2,
we determine the whole distinguished component, which is an affine
homogeneous $\underline{G}/\underline{H}$ with
$\rank_{\underline{G}}(\underline{G}/\underline{H})=\rank\underline{G}$,
$\underline{G}:=N_G(L_0)^\circ/L_0, \underline{H}:=(H\cap
N_G(L_0)^\circ)/L_0$. We know that
$\X_{G,G/H}=\X_{\underline{G},\underline{G}/\underline{H}}$ (see
Proposition \ref{Prop:1.5.3}). The last lattice is computed as in
case 1. A point from the distinguished component of
$(\underline{G}/\underline{H})^{L_{0\,\underline{G},\underline{G}/\underline{H}}}$
lies in the distinguished component of $(G/H)^{L_{0\,G,G/H}}$.

{\it Case 3.} Here $X=G*_HV$ is an affine homogeneous vector bundle
and $\pi:G*_HV\rightarrow G/H$ is the natural projection. Applying
the algorithm of case 2 to $G/H$, we compute the lattice
$\X_{G,G/H}$ and find a point $x$ in the distinguished component of
$(G/H)^{L_{0\,G,G/H}}$. Applying the following algorithm to the
group $L_0:=L_{0\,G,G/H}$ and the $L_0$-module $V:=\pi^{-1}(x)$, we
compute $L_{0\,L_0,V}$.

\begin{Alg}\label{Alg:9.1.1}  Set $G_0=L_0, V_0=V$.
Assume that we have already constructed a pair $(G_i,V_i)$, where
$G_i$ is an almost connected connected subgroup in $G_0$ and $V_i$
is a $G_i$-module. Set $\widetilde{B}_i:=B\cap G_i,
B_i:=\widetilde{B}_i^\circ$. Choose a
$\widetilde{B}_i$-semiinvariant vector $\alpha\in V_i^*$. Put
$V_{i+1}:=(\u_i^-\alpha)^0$, where $\u_i^-$ is a maximal unipotent
subalgebra of $\g_i$ normalized by $T$ and opposite to $\b_i$ and
the superscript $^0$ means the annihilator. Put
$G_{i+1}:=Z_{G_i}(\alpha)$. The group $G_{i+1}$ is almost connected
and $L_{0\,G_i,V_i}=L_{0\,G_{i+1},V_{i+1}}$. Note that $\rank
[\g_{i+1},\g_{i+1}]\leqslant \rank [\g_i,\g_i]$ with the equality
iff $\alpha\in V^{[\g_i,\g_i]}$. Thus if $[\g_i,\g_i]$ acts
non-trivially on $V$, then we may assume that $\rank
[\g_{i+1},\g_{i+1}]< \rank [\g_i,\g_i]$. So $V_k:=V_k^{[\g_k,\g_k]}$
for some $k$. Here $L_{0\,L_0,V}=L_{0\,G_k,V_k}$ coincides with the
inefficiency kernel for the action $G_k:V_k$.
\end{Alg}

By Proposition \ref{Prop:1.5.8}, $L_{0\,G,X}=L_{0\,L_0,V}$.

{\it Case 4.} Suppose $X=G/H$, where $H$ is a nonreductive subgroup
of $G$. We find a parabolic subgroup $Q\subset G$ tamely containing
$H$ by using Algorithm 7.1.2 from \cite{Weyl}. Further, we choose a
Levi subgroup $M\subset Q$ and  $g\in G$ such that $M\cap H$ is a
maximal reudctive subgroup of $H$ and $gQg^{-1}$ is an antistandard
parabolic subgroup and $gMg^{-1}$ is its standard Levi subgroup.
Replace $(Q,M,H)$ with $(gQg^{-1},gMg^{-1},gHg^{-1})$. Put
$X':=Q^-/H$. Using Remark 3.2.8 from \cite{Weyl}, we construct an
$M$-isomorphism of $X'$ with an affine homogeneous vector bundle. By
Proposition \ref{Prop:1.5.4}, $\X_{G,G/H}=\X_{M,X'}$. The last
lattice is computed as in case 3.

\bigskip
{\Small Chair of Higher Algebra, Department of Mechanics and
Mathematics, Moscow State University.

\noindent E-mail address: ivanlosev@yandex.ru}
\end{document}